\documentclass[a4paper,british]{scrartcl}
\usepackage[T1]{fontenc}
\usepackage[latin9]{inputenc}
\usepackage{babel}
\usepackage{prettyref}
\usepackage{mathtools}
\usepackage{amsmath}
\usepackage{amsthm}
\usepackage{amssymb}
\usepackage{esint}
\usepackage[unicode=true,pdfusetitle,
 bookmarks=true,bookmarksnumbered=false,bookmarksopen=false,
 breaklinks=false,pdfborder={0 0 1},backref=false,colorlinks=false]
 {hyperref}

\makeatletter

\pdfpageheight\paperheight
\pdfpagewidth\paperwidth

\theoremstyle{definition}
\newtheorem*{defn*}{\protect\definitionname}
\theoremstyle{plain}
\newtheorem{thm}{\protect\theoremname}[section]
\theoremstyle{remark}
\newtheorem{rem}[thm]{\protect\remarkname}
\theoremstyle{plain}
\newtheorem{prop}[thm]{\protect\propositionname}
\theoremstyle{plain}
\newtheorem{lem}[thm]{\protect\lemmaname}
\theoremstyle{definition}
\newtheorem{example}[thm]{\protect\examplename}
\theoremstyle{plain}
\newtheorem{cor}[thm]{\protect\corollaryname}

\@ifundefined{date}{}{\date{}}
\usepackage{prettyref}

\usepackage{enumerate}

\newcommand*{\e}{\mathrm{e}}
\renewcommand*{\i}{\mathrm{i}}

\newcommand{\R}{\mathbb{R}}
\newcommand{\N}{\mathbb{N}}
\newcommand{\dom}{\operatorname{dom}}
\newcommand{\ran}{\operatorname{ran}}
\newcommand{\spt}{\operatorname{spt}}
\renewcommand{\d}{\,\mathrm{d}}
\renewcommand{\Re}{\operatorname{Re}}
\newcommand{\curl}{\operatorname{curl}}
\newcommand{\dive}{\operatorname{div}}
\renewcommand{\tilde}{\widetilde}

\newrefformat{prop}{Proposition \ref{#1}}
\newrefformat{lem}{Lemma \ref{#1}}
\newrefformat{thm}{Theorem \ref{#1}}
\newrefformat{cor}{Corollary \ref{#1}}
\newrefformat{rem}{Remark \ref{#1}}
\newrefformat{exa}{Example \ref{#1}}
\newrefformat{sub}{Subsection \ref{#1}}
\newrefformat{eq}{(\ref{#1})}

\theoremstyle{definition}
\newtheorem{hyp}{Hypotheses}

\newrefformat{hyp}{Hypotheses \ref{#1}}

\makeatother

\providecommand{\corollaryname}{Corollary}
\providecommand{\definitionname}{Definition}
\providecommand{\examplename}{Example}
\providecommand{\lemmaname}{Lemma}
\providecommand{\propositionname}{Proposition}
\providecommand{\remarkname}{Remark}
\providecommand{\theoremname}{Theorem}

\begin{document}
\title{Well-posedness for a general class of differential inclusions}
\author{Sascha Trostorff\thanks{Mathematisches Seminar, CAU Kiel, Germany, e-mail: trostorff@math.uni-kiel.de}}

\maketitle
\textbf{Abstract. }We consider an abstract class of differential inclusions,
which covers differential-algebraic and non-autonomous problems as
well as problems with delay. Under weak assumptions on the operators
involved, we prove the well-posedness of those differential inclusions
in a pure Hilbert space setting. Moreover, we study the causality
of the associated solution operator. The theory is illustrated by
an application to a semistatic quasilinear variant of Maxwell's equations.\medskip{}

\textbf{Keywords: }differential inclusions, maximal monotone relations,
non-autonomous problems, causality, Maxwell's equations\textbf{\medskip{}
}

\textbf{2010 MSC: }34G25, 46N20, 35F60\textbf{\medskip{}
}

\section{Introduction }

In \cite{Picard} it was shown that most (if not all) linear autonomous
equations in mathematical physics share a common form, namely 
\[
(\partial_{t}\mathcal{M}+\mathcal{A})U=F,
\]
where $\partial_{t}$ stands for the derivative with respect to time,
$\mathcal{M}$ is a suitable bounded operator in space-time, which
commutes with $\partial_{t}$ and $\mathcal{A}$ is a skew-selfadjoint
operator in space, which in applications is a spatial differential
operator. This result was generalised by the author in \cite{Trostorff2012_NA,Trostorff2012_nonlin_bd,Trostorff_2011}
to the case of differential inclusions, where the operator $\mathcal{A}$
is replaced by a maximal monotone relation. The resulting problem
then takes the form 
\[
(U,F)\in\partial_{t}\mathcal{M}+\mathcal{A}.
\]
This generalisation allows to study certain non-linear problems, especially
equations describing physical phenomena with hysteresis effects. A
particular case of an operator $\mathcal{M}$ is given by $\mathcal{M}=M_{0}+\partial_{t}^{-1}M_{1}$
for some bounded spatial operators $M_{0},M_{1}.$ The resulting problem
then takes the simpler form 
\[
(U,F)\in\partial_{t}M_{0}+M_{1}+\mathcal{A}.
\]
Replacing now $M_{0},M_{1}$ by operator-valued multiplication operators,
the resulting problem becomes non-autonomous. These problems were
studied in \cite{Picard2013_nonauto} for the case of skew-selfadjoint
operators $\mathcal{A}$ and in \cite{Trostorff2013_nonautoincl}
for the case of a maximal monotone relation $\mathcal{A}.$ However,
in the case of maximal monotone relations, the authors of \cite{Trostorff2013_nonautoincl}
had to restrict the class of admissible multiplication operators $M_{0}$
and $M_{1}.$ Finally, in \cite{Waurick2015_nonauto} Waurick proved
a well-posedness result for a very abstract class of non-autonomous
differential equations, where $M_{0}$ and $M_{1}$ are replaced by
arbitrary space-time operators $\mathcal{M},\mathcal{N},$ where $\mathcal{M}$
should have a bounded commutator with $\partial_{t}$. With this result
he was able to generalise both the results of \cite{Picard} and \cite{Picard2013_nonauto}.
However, it does not cover the result for the inclusions. It is the
purpose of that article to provide a solution theory for differential
inclusions of the form 
\begin{equation}
(U,F)\in\partial_{t}\mathcal{M}+\mathcal{N}+\mathcal{A},\label{eq:prob}
\end{equation}
where $\mathcal{M}$ and $\mathcal{N}$ are space-time operators and
$\mathcal{A}$ is a maximal monotone relation. We hereby generalise
the results of \cite{Trostorff2012_nonlin_bd,Trostorff2013_nonautoincl}
and \cite{Waurick2015_nonauto} and provide a unified solution theory
for a broad class of problems.

Of course the problem of non-autonomous differential inclusions was
studied in the literature before. We just mention some classical approaches,
to tackle this problem. A standard approach for inclusions given in
form of a Cauchy problem (i.e. $\mathcal{M}=1$ and time dependent
$\mathcal{A}$ in \prettyref{eq:prob}) uses the concept of evolution
families introduced by Kato \cite{Kato1953} for evolution equations
and generalised to inclusions by Crandall and Pazy in \cite{Crandall1972}.
Other approaches use approximations by replacing the time derivative
in \prettyref{eq:prob} by the difference quotient and then proving
that the corresponding solutions converge in a suitable sense (e.g.
\cite{Evans1977,Kobayasi1984}). A third approach uses the notion
of integral solutions introduced by Bénilan in \cite{Benilan1972}
for autonomous problems and generalised in \cite{Kobayasi1984} to
the non-autonomous case. 

In contrast to all these classical approaches, we assume the relation
$\mathcal{A}$ to be time-independent. The non-autonomous behaviour
enters the equation via the bounded space-time operators $\mathcal{M}$
and $\mathcal{N}.$ This has the big advantage that we do not have
any problems with time-dependent domains. 

The article is structured as follows: In the next section we present
the underlying Hilbert space framework and recall some basic facts
of maximal monotone relations. Section 3 is devoted to the well-posedness
result for inclusions of the form \prettyref{eq:prob} in an exponentially
weighted $L_{2}$-space. In Section 4 we prove under stronger assumptions
the causality of the solution operator and its independence of the
particular choice of the exponential weight (see \prettyref{thm:independence}
for the precise statement). Finally, we apply our results to a semistatic
quasilinear variant of Maxwell's equations and thereby generalise
the result of \cite{Milani_Picard1995}.

\section{Preliminaries}

We begin by introducing the Hilbert space setting we are working with
throughout this article. Throughout let $H$ denote a Hilbert space
(real or complex, in case of a real Hilbert space one can ignore all
forthcoming occurring real parts) with inner product $\langle\cdot,\cdot\rangle$
and induced norm $|\cdot|.$ Following \cite{Picard_McGhee,Picard}
we introduce the following weighted $L_{2}$-space and the derivative
operator defined on it.
\begin{defn*}
Let $\rho\geq0.$ We define the following space of (equivalence classes
of) square integrable functions with respect to an exponentially weighted
Lebesgue measure 
\[
L_{2,\rho}(\R;H)\coloneqq\left\{ f:\R\to H\,;\,f\text{ Bochner-measurable, }\intop_{\R}|f(t)|^{2}\,\e^{-2\rho t}\d t<\infty\right\} .
\]
Clearly, $L_{2,\rho}(\R;H)$ becomes a Hilbert space equipped with
the usual inner product 
\[
\langle f,g\rangle_{\rho}\coloneqq\intop_{\R}\langle f(t),g(t)\rangle\,\e^{-2\rho t}\d t\quad(f,g\in L_{2,\rho}(\R;H)).
\]
We denote the induced norm by $|\cdot|_{\rho}.$ Moreover, we define
the operator $\partial_{t,\rho}$ as the closure of the operator 
\[
C_{c}^{\infty}(\R;H)\subseteq L_{2,\rho}(\R;H)\to L_{2,\rho}(\R;H):\,\phi\mapsto\phi',
\]
where $C_{c}^{\infty}(\R;H)$ denotes the space of $H$-valued arbitrarily
differentiable functions on $\R$ with compact support. 
\end{defn*}
\begin{rem}
\label{rem:derivative}\begin{enumerate}[(a)]

\item The domain of $\partial_{t,\rho}$ consists of those elements
in $L_{2,\rho}(\R;H)$ whose distributional derivative lies again
in $L_{2,\rho}(\R;H).$

\item For $\rho=0$ the space $L_{2,0}(\R;H)$ is the usual $L_{2}$-space
and the derivative $\partial_{t,0}$ coincides with the classical
weak derivative with domain $H^{1}(\R;H).$

\end{enumerate}
\end{rem}

We recall some basic properties of $\partial_{t,\rho}$ and refer
to \cite{Kalauch2011,Picard2014_survey} for the respective proofs.
\begin{prop}
Let $\rho\geq0.$ Then the operator $\partial_{t,\rho}$ is normal
with $\partial_{t,\rho}^{\ast}=-\partial_{t,\rho}+2\rho.$ Moreover,
$\sigma(\partial_{t,\rho})=\left\{ \i t+\rho\,;\,t\in\R\right\} $.
In particular, if $\rho>0,$ then $\partial_{t,\rho}$ is boundedly
invertible with $\|\partial_{t,\rho}^{-1}\|=\frac{1}{\rho}$ and 
\[
\left(\partial_{t,\rho}^{-1}f\right)(t)=\intop_{-\infty}^{t}f(s)\d s\quad(f\in L_{2,\rho}(\R;H),t\in\R).
\]
\end{prop}

\begin{rem}
It should be noted that the bounded invertibility of $\partial_{t,\rho}$
just holds, since we deal with the whole real line and not just with
the positive real line or an interval. The main reason is that on
the positive real line or an interval, we need to impose initial conditions
to obtain invertibility of $\partial_{t,\rho}$. By dealing with the
whole real line, we implicitly impose a vanishing initial condition
at $-\infty$. For the treatment of differential inclusions on $\mathbb{R}_{\geq0}$
within the framework introduced here, we refer to \cite{Trostorff2012_NA}.
\end{rem}

With this operator at hand, we are able to define a scale of associated
Hilbert spaces called the Sobolev chain associated with $\partial_{t,\rho}$,
\cite{Picard_McGhee} (see also \cite{Nagel1997,engel2000one}, where
these spaces are called Sobolev towers). 
\begin{prop}
Let $\rho>0.$ For $k\in\N$ we define the spaces 
\[
H_{\rho}^{k}(\R;H)\coloneqq\dom(\partial_{t,\rho}^{k})
\]
equipped with the inner product 
\[
\langle u,v\rangle_{\rho,k}\coloneqq\langle\partial_{t,\rho}^{k}u,\partial_{t,\rho}^{k}v\rangle_{\rho}\quad(u,v\in\dom(\partial_{t,\rho}^{k})).
\]
Moreover, we set $H_{\rho}^{-k}(\R;H)$ as the completion of $L_{2,\rho}(\R;H)$
with respect to the norm induced by 
\[
\langle u,v\rangle_{\rho,-k}\coloneqq\langle\partial_{t,\rho}^{-k}u,\partial_{t,\rho}^{-k}v\rangle_{\rho}\quad(u,v\in L_{2,\rho}(\R;H)).
\]
For $k,j\in\mathbb{Z}$ with $k\geq j$ we have that 
\[
H_{\rho}^{k}(\R;H)\hookrightarrow H_{\rho}^{j}(\R;H)
\]
and the operator 
\[
\partial_{t,\rho}^{k-j}:C_{c}^{\infty}(\R;H)\subseteq H_{\rho}^{k}(\R;H)\to H_{\rho}^{j}(\R;H)
\]
extends to a unitary operator again denoted by $\partial_{t,\rho}^{k-j}.$
\end{prop}

We state a simple observation, which will be used several times later
on.
\begin{lem}
\label{lem:better_limit}Let $\rho>0$ and $k,j\in\mathbb{Z}$ with
$k\geq j.$ Let $(u_{n})_{n\in\mathbb{N}}$ be a sequence in $H_{\rho}^{k}(\R;H)$
such that $u_{n}\to u$ in $H_{\rho}^{j}(\R;H)$ for some $u\in H_{\rho}^{j}(\R;H)$
and assume that $(u_{n})_{n}$ is bounded in $H_{\rho}^{k}(\R;H).$
Then $u\in H_{\rho}^{k}(\R;H).$
\end{lem}

\begin{proof}
As $(u_{n})_{n}$ is bounded in $H_{\rho}^{k}(\R;H),$ we can assume
without loss of generality that $u_{n}\rightharpoonup v$ for some
$v\in H_{\rho}^{k}(\R;H).$ As $H_{\rho}^{k}(\R;H)\hookrightarrow H_{\rho}^{j}(\R;H)$
we infer that $u_{n}\rightharpoonup v$ in $H_{\rho}^{j}(\R;H)$ and
hence, $u=v\in H_{\rho}^{k}(\R;H).$
\end{proof}
Finally, we provide a useful characterisation for elements lying in
$H_{\rho}^{1}(\R;H).$
\begin{lem}[{\cite[Proposition 2.1]{Trostorff2013_nonautoincl}}]
\label{lem:difference_quotient} Let $\rho\geq0$, $u\in L_{2,\rho}(\R;H)$
and for $h>0$ we define 
\[
D_{h}u\coloneqq\frac{1}{h}(\tau_{h}u-u)\in L_{2,\rho}(\R;H),
\]
where $\tau_{h}u\coloneqq u(\cdot+h).$ Then, $u\in H_{\rho}^{1}(\R;H)$
if and only if $(D_{h}u)_{h\in]0,1]}$ is bounded in $L_{2,\rho}(\R;H).$
In each case we have that 
\[
D_{h}u\to\partial_{t,\rho}u\quad(h\to0+)
\]
in $L_{2,\rho}(\R;H).$ 
\end{lem}

Besides the presented Hilbert space setting, we need the framework
of maximal monotone relations. For this topic and the proofs of the
subsequent results we refer to the monographs \cite{Brezis,Morosanu,papageogiou}.
\begin{defn*}
A binary relation $A\subseteq H\times H$ is called \emph{monotone},
if 
\[
\forall(u,v),(x,y)\in A:\:\Re\langle u-x,v-y\rangle\geq0.
\]

$A$ is called \emph{maximal monotone} if $A$ is monotone and for
each monotone relation $B\subseteq H\times H$ we have that 
\[
A\subseteq B\Rightarrow A=B.
\]

Moreover, we say that $A$ is \emph{$c$-(maximal) monotone} for some
$c\in\R,$ if $A-c\coloneqq\{(u,v-cu)\,;\,(u,v)\in A\}$ is (maximal)
monotone.
\end{defn*}
We need the following lifting result.
\begin{lem}[{\cite[Example 2.3.3]{Brezis}}]
\label{lem:lifting}Let $A\subseteq H\times H$ and $(\Omega,\Sigma,\mu)$
a measure space. We define 
\[
A_{\mu}\coloneqq\left\{ (u,v)\in L_{2}(\mu;H)\times L_{2}(\mu;H)\,;\,(u(t),v(t))\in A\quad(t\in\Omega\text{ a.e.})\right\} .
\]
If $A$ is monotone, then so is $A_{\mu}.$ Moreover, if $A$ is maximal
monotone and $(0,0)\in A$ it follows that $A_{\mu}$ is maximal monotone
as well.
\end{lem}

\begin{rem}
If $\mu(\Omega)<\infty$ then the assumption $(0,0)\in A$ can be
dropped. However, since we want to apply this lifting result to $\Omega=\mathbb{R}$
and $\mu=\e^{-2\rho t}\d t$, we need to impose this condition in
our main result. If one just deals with $\mathbb{R}_{\geq0}$ as time
horizon, one obtains $\mu(\mathbb{R}_{\geq0})=\frac{1}{2\rho}<\infty$
and hence, this additional condition is superfluous.
\end{rem}

The maximal monotonicity of a monotone relation $A$ can be characterised
by a range condition on $A$. This is the celebrated Theorem by Minty.
\begin{thm}[Minty, \cite{Minty}]
\label{thm:Minty} Let $A\subseteq H\times H$ be monotone. Then
the following statements are equivalent:

\begin{enumerate}[(i)]

\item $A$ is maximal monotone,

\item For each $\lambda>0$ and each $z\in H$ there exists $u\in H$
such that $(u,z)\in1+\lambda A,$ i.e. there is $v\in H$ with $(u,v)\in A$
and $u+\lambda v=z.$ 

\item There exists $\lambda>0$ such that for each $z\in H$ there
exists $u\in H$ with $(u,z)\in1+\lambda A.$

\end{enumerate}
\end{thm}

\begin{example}
\label{exa:Time_derivative}For $\rho>0$ the operator $\partial_{t,\rho}$
is $\rho$-maximal monotone. The monotonicity follows by 
\[
\langle\partial_{t,\rho}\phi,\phi\rangle_{\rho}=\langle\phi,\partial_{t,\rho}^{\ast}\phi\rangle_{\rho}=\langle\phi,-\partial_{t,\rho}\phi\rangle_{\rho}+2\rho|\phi|_{\rho}^{2}\quad(\phi\in\dom(\partial_{t,\rho})),
\]
which yields 
\[
\Re\langle\partial_{t,\rho}\phi,\phi\rangle_{\rho}\geq\rho|\phi|_{\rho}^{2}\quad(\phi\in\dom(\partial_{t,\rho})).
\]
Moreover, for $f\in L_{2,\rho}(\R;H)$ we set 
\[
u\coloneqq\rho\partial_{t,\rho}^{-1}f\in\dom(\partial_{t,\rho})
\]
and obtain 
\[
u+\frac{1}{\rho}(\partial_{t,\rho}-\rho)u=\frac{1}{\rho}\partial_{t,\rho}u=f,
\]
which shows the maximal monotonicity by \prettyref{thm:Minty}. 
\end{example}

As a consequence of Minty's theorem and the definition of monotonicity
we obtain the following proposition.
\begin{prop}[{\cite[Proposition 2.6]{Brezis}}]
 Let $A\subseteq H\times H$ be maximal monotone. Then $(1+\lambda A)^{-1}$
is a Lipschitz-continuous mapping defined on the whole Hilbert space
$H$ with\footnote{We denote the smallest Lipschitz-constant for a Lipschitz continuous
mapping $f$ by $|f|_{\mathrm{Lip}}.$} $|(1+\lambda A)^{-1}|_{\mathrm{Lip}}\leq1.$ Moreover, for each $\lambda>0$
the so-called \emph{Yosida-approximation}
\[
A_{\lambda}\coloneqq\frac{1}{\lambda}(1-(1+\lambda A)^{-1}):H\to H
\]
is a monotone and Lipschitz-continuous mapping with $|A_{\lambda}|_{\mathrm{Lip}}\leq\frac{1}{\lambda}.$ 
\end{prop}

We conclude this section by some well-known perturbation results for
maximal monotone relations.
\begin{prop}
\label{prop:pert}Let $A\subseteq H\times H$ be maximal monotone
and $B:H\to H$ Lipschitz continuous. Moreover, assume that $A+B\coloneqq\left\{ (x,y+Bx)\,;\,(x,y)\in A\right\} $
is monotone. Then $A+B$ is maximal monotone.
\end{prop}

\begin{proof}
For $B=0$ there is nothing to show. So assume that $B\ne0.$ By \prettyref{thm:Minty}
it suffices to prove that for $0<\lambda<\frac{1}{|B|_{\mathrm{Lip}}}$
and each $z\in H$ there exists $u\in H$ with $(u,z)\in1+\lambda(A+B).$
The latter is equivalent to 
\[
u=(1+\lambda A)^{-1}(z-\lambda Bu).
\]
Since $(1+\lambda A)^{-1}$ is Lipschitz-continuous with $|(1+\lambda A)^{-1}|_{\mathrm{Lip}}\leq1$
and $|\lambda B|_{\mathrm{Lip}}<1$ by the choice of $\lambda,$ the
assertion follows from the contraction mapping theorem.
\end{proof}
\begin{cor}
Let $B:H\to H$ be monotone and Lipschitz-continuous. Then $B$ is
maximal monotone.
\end{cor}

In particular, if $B\subseteq H\times H$ is maximal monotone, then
its Yosida-approximation $B_{\lambda}$ is Lipschitz-continuous and
monotone and hence, maximal monotone. Moreover, if $A\subseteq H\times H$
is maximal monotone, then so is $A+B_{\lambda}$ for each $\lambda>0$
and hence, $(1+A+B_{\lambda})^{-1}$ is a Lipschitz-continuous mapping.
This observation can be used to prove the following perturbation result.
\begin{prop}[{\cite[Proposition 3.1]{papageogiou}}]
\label{prop:pert-yosida} Let $A,B\subseteq H\times H$ be two maximal
monotone relations. Moreover, let $z\in H.$ Then there exists $u\in H$
such that $(u,z)\in1+A+B$ (i.e., there are $v,w\in H$ with $(u,v)\in A,(u,w)\in B$
and $u+v+w=z$) if and only if 
\[
\sup_{\lambda>0}\left|B_{\lambda}(u_{\lambda})\right|<\infty
\]
where $u_{\lambda}\coloneqq(1+A+B_{\lambda})^{-1}(z).$ In the latter
case, $u_{\lambda}\to u$ as $\lambda\to0+.$
\end{prop}

With the help of the latter proposition, one can prove the following
perturbation result.
\begin{cor}[{\cite[Proposition 1.22]{Trostorff_2011}}]
\label{cor:bounded_pert} Let $A,B\subseteq H\times H$ be two maximal
monotone relations. Moreover, assume that $A$ is \emph{bounded},
i.e. for each $U\subseteq H$ bounded, the set 
\[
\left\{ v\in H\,;\,\exists u\in U:(u,v)\in A\right\} 
\]
is bounded. If $\dom(A)\cap\dom(B)\ne\emptyset$ (i.e. $\exists u,v,w\in H:\,(u,v)\in A$
and $(u,w)\in B$) then $A+B$ is maximal monotone. 
\end{cor}

\section{The main result}

Throughout, let $\rho>0.$ We begin with stating the hypotheses of
the operators and relations involved, which we assume to be valid
throughout this section. 

\begin{hyp}\label{hyp:assumptions} Let $\mathcal{M},\mathcal{N},\mathcal{M}'\in L(L_{2,\rho}(\R;H))$
be such that:

\begin{enumerate}[(a)]

\item $\mathcal{M}\partial_{t,\rho}\subseteq\partial_{t,\rho}\mathcal{M}-\mathcal{M}',$

\item There exists $c>0$ such that 
\[
\Re\langle(\partial_{t,\rho}\mathcal{M}+\mathcal{N})\varphi,\varphi\rangle_{\rho}\geq c|\varphi|_{\rho}^{2}
\]
for each $\varphi\in C_{c}^{\infty}(\R;H).$

\end{enumerate}

Moreover, let $\mathcal{A}\subseteq L_{2,\rho}(\R;H)\times L_{2,\rho}(\R;H)$
be maximal monotone such that 
\[
\forall u,v\in L_{2,\rho}(\R;H):\,\left((u,v)\in\mathcal{A}\Rightarrow\forall h\geq0:(\tau_{h}u,\tau_{h}v)\in\mathcal{A}\right).
\]

\end{hyp}

It is noteworthy that the perspective on differential inclusions presented
here, is that the relation $\mathcal{A}$ is independent of time in
the sense that it commutes with time translation, while the time-dependence
enter the problem via coefficients, which are incorporated in the
operators $\mathcal{M}$ and $\mathcal{N}.$ This is the standard
case in problems arising in mathematical physics.

Our main theorem reads as follows.
\begin{thm}[Well-posedness]
\label{thm:main} The relation $\overline{\partial_{t,\rho}\mathcal{M}+\mathcal{N}+\mathcal{A}}$
is $c$-maximal monotone. In particular, the inverse relation 
\[
\mathcal{S}_{\rho}\coloneqq\left(\overline{\partial_{t,\rho}\mathcal{M}+\mathcal{N}+\mathcal{A}}\right)^{-1}:L_{2,\rho}(\R;H)\to L_{2,\rho}(\R;H)
\]
is a Lipschitz-continuous mapping with $\left|\mathcal{S}_{\rho}\right|_{\mathrm{Lip}}\leq\frac{1}{c}$.
Hence, for each $f\in L_{2,\rho}(\R;H)$ there exists a unique $u\in L_{2,\rho}(\R;H)$
with 
\[
(u,f)\in\overline{\partial_{t,\rho}\mathcal{M}+\mathcal{N}+\mathcal{A}}
\]
and $u$ depends Lipschitz-continuously on $f$.
\end{thm}

Before we come to the proof of this theorem, we illustrate the hypotheses
by some concrete examples for the operators $\mathcal{M}$ and $\mathcal{N}$.
\begin{example}
\begin{enumerate}[(a)]

\item If $\mathcal{N}=0$ and $\mathcal{M}$ is given in terms of
an operator-valued analytic function of $\partial_{t,\rho}$, then
condition (a) is trivially satisfied with $\mathcal{M}'=0$ and the
resulting problem becomes autonomous. This class of operators was
introduced in \cite{Picard} and allows the treatment of a broad class
of differential equations arising in mathematical physics (see e.g.
\cite{Picard,Picard2010_poroelastic,Picard2012_Impedance,Picard2015_micropoloar})
and includes different types of equations, such as fractional differential
equations (see \cite{Picard2013_fractional}), delay equations with
finite and infinite delay (see \cite{Kalauch2011}), and integro-differential
equations (see \cite{Trostorff2012_integro}).

\item A second class of problems covered by the hypotheses is given
by non-autonomous equations, which are local with respect to time.
More precisely, $\mathcal{M},\mathcal{N}$ are given as operator-valued
multiplication operators, such that $\mathcal{M}$ is Lipschitz-continuous.
Then, Rademacher's theorem implies that (a) holds. Such equations
were studied in \cite{Picard2013_nonauto} and generalised to inclusions
in \cite{Trostorff2013_nonautoincl}.

\item The hypotheses also allow the treatment of non-autonomous differential
inclusions which are non-local in space and time. A classical example
would be integral operators of the form
\[
\left(\mathcal{M}u\right)(t)\coloneqq\int_{\mathbb{R}}k(t,s)u(s)\d s\quad(t\in\mathbb{R})
\]
with a suitable (possibly operator-valued) kernel $k$. Indeed, if
$k$ is differentiable with respect to the first variable and satisfies
suitable integrability conditions, one easily can show that (a) holds.
To find the right conditions on $k$ to ensure (b) is more delicate
and will be postponed to future research. In case of a kernel $k(t,s)=k(t-s)$,
this was done in \cite{Trostorff2012_integro,Trostorff_habil} even
for operator-valued kernels.

\end{enumerate}
\end{example}

We begin by proving \prettyref{thm:main} in the case $\mathcal{A}=0.$
The proof follows the rationale of \cite[Section 3.3]{Waurick_habil}.
For the readers convenience we recall the definition of a core of
a closed operator.
\begin{defn*}
Let $S:\dom(S)\subseteq X\to Y$ be a closed linear operator between
two normed spaces $X$ and $Y$. A linear subspace $D\subseteq X$
is called a \emph{core for $S$}, if $S=\overline{S|_{D}}.$
\end{defn*}
\begin{prop}
\label{prop:max_mon_materiallaw} $C_{c}^{\infty}(\R;H)$ is a core
for $\partial_{t,\rho}\mathcal{M}$. Moreover, the operator $\partial_{t,\rho}\mathcal{M}+\mathcal{N}$
is $c$-maximal monotone.
\end{prop}

\begin{proof}
First, we observe that $(1+\varepsilon\partial_{t,\rho})^{-1}\to1$
strongly as $\varepsilon\to0+.$ Indeed, since $\partial_{t,\rho}$
is maximal monotone by \prettyref{exa:Time_derivative} we have $\|(1+\varepsilon\partial_{t,\rho})^{-1}\|\leq1$
for each $\varepsilon>0.$ Thus, it suffices to prove the strong convergence
on a dense subset of $L_{2,\rho}(\R;H).$ Since for $u\in\dom(\partial_{t,\rho})$
we have that 
\[
(1+\varepsilon\partial_{t,\rho})^{-1}u-u=-\varepsilon(1+\varepsilon\partial_{t,\rho})^{-1}\partial_{t,\rho}u\to0\quad(\varepsilon\to0+)
\]
the assertion follows.\\
We now prove that $\dom(\partial_{t,\rho})$ is a core for $\partial_{t,\rho}\mathcal{M}$.
More precisely, we show that for $u\in\dom(\partial_{t,\rho}\mathcal{M})$
and $u_{\varepsilon}\coloneqq(1+\varepsilon\partial_{t,\rho})^{-1}u$
with $\varepsilon>0$ we have that 
\[
\partial_{t,\rho}\mathcal{M}u_{\varepsilon}\to\partial_{t,\rho}\mathcal{M}u\quad(\varepsilon\to0+).
\]
Indeed, we have that 
\begin{align*}
\partial_{t,\rho}\mathcal{M}u_{\varepsilon} & =\partial_{t,\rho}\mathcal{M}(1+\varepsilon\partial_{t,\rho})^{-1}u\\
 & =(1+\varepsilon\partial_{t,\rho})^{-1}\partial_{t,\rho}\mathcal{M}u+\varepsilon\partial_{t,\rho}(1+\varepsilon\partial_{t,\rho})^{-1}\mathcal{M}'(1+\varepsilon\partial_{t,\rho})^{-1}u\\
 & =(1+\varepsilon\partial_{t,\rho})^{-1}\partial_{t,\rho}\mathcal{M}u+\mathcal{M}'(1+\varepsilon\partial_{t,\rho})^{-1}u-(1+\varepsilon\partial_{t,\rho})^{-1}\mathcal{M}'(1+\varepsilon\partial_{t,\rho})^{-1}u\\
 & \to\partial_{t,\rho}\mathcal{M}u\quad(\varepsilon\to0+).
\end{align*}

Hence, it suffices to approximate elements $u\in\dom(\partial_{t,\rho})$
by a sequence $(\varphi_{n})_{n}$ in $C_{c}^{\infty}(\R;H)$ such
that $\varphi_{n}\to u$ and $\partial_{t,\rho}\mathcal{M}\varphi_{n}\to\partial_{t,\rho}\mathcal{M}u$
in $L_{2,\rho}(\R;H).$ For doing so, we choose a sequence $(\varphi_{n})_{n}$
in $C_{c}^{\infty}(\R;H)$ such that $\varphi_{n}\to u$ in $H_{\rho}^{1}(\R;H)$.
Then, in particular, $\varphi_{n}\to u$ in $L_{2,\rho}(\R;H)$ and
\[
\partial_{t,\rho}\mathcal{M}\varphi_{n}=\mathcal{M}\partial_{t,\rho}\varphi_{n}+\mathcal{M}'\varphi_{n}\to\mathcal{M}\partial_{t,\rho}u+\mathcal{M}'u=\partial_{t,\rho}\mathcal{M}u.
\]
We now prove the maximal monotonicity of $\partial_{t,\rho}\mathcal{M}+\mathcal{N}-c$.
As $C_{c}^{\infty}(\R;H)$ is a core for this operator, the monotonicity
follows by \prettyref{hyp:assumptions} (b). We claim that $C_{c}^{\infty}(\R;H)$
is also a core for $\left(\partial_{t,\rho}\mathcal{M}+\mathcal{N}\right)^{\ast}$.
Assuming that this is true, it follows from \prettyref{hyp:assumptions}
(b) that $\left(\partial_{t,\rho}\mathcal{M}+\mathcal{N}\right)^{\ast}$
is one-to-one and hence, $\partial_{t,\rho}\mathcal{M}+\mathcal{N}$
has dense range. Since $\partial_{t,\rho}\mathcal{M}+\mathcal{N}$
is also closed and its inverse is bounded by what we have shown before,
$\partial_{t,\rho}\mathcal{M}+\mathcal{N}$ is indeed onto and thus,
$\partial_{t,\rho}\mathcal{M}+\mathcal{N}-c$ is maximal monotone.
Thus, we are left to show that $C_{c}^{\infty}(\R;H)$ is a core for
$\left(\partial_{t,\rho}\mathcal{M}+\mathcal{N}\right)^{\ast}$. For
showing this, we observe that
\[
\left(\partial_{t,\rho}\mathcal{M}+\mathcal{N}\right)^{\ast}=(\partial_{t,\rho}\mathcal{M})^{\ast}+\mathcal{N}^{\ast}=(\mathcal{M}\partial_{t,\rho}+\mathcal{M}')^{\ast}+\mathcal{N}^{\ast}=\partial_{t,\rho}^{\ast}\mathcal{M}^{\ast}+(\mathcal{M}')^{\ast}+\mathcal{N}^{\ast},
\]
where we have used that $C_{c}^{\infty}(\R;H)$ is a core for $\partial_{t,\rho}\mathcal{M}.$
Thus, it suffices to show that $C_{c}^{\infty}(\R;H)$ is a core for
$\partial_{t,\rho}^{\ast}\mathcal{M}^{\ast}$. However, since 
\[
\mathcal{M}^{\ast}\partial_{t,\rho}^{\ast}\subseteq(\partial_{t,\rho}\mathcal{M})^{\ast}=\partial_{t,\rho}^{\ast}\mathcal{M}^{\ast}+(\mathcal{M}')^{\ast}
\]
we can follow the same lines as above and obtain the assertion. 
\end{proof}
In order to prove \prettyref{thm:main} we adopt the idea presented
in \cite{Trostorff2013_nonautoincl} and first prove the well-posedness
of an auxiliary problem. The well-posedness of the original problem
will then follow by the perturbation result \prettyref{prop:pert}.
The auxiliary problem reads as follows
\begin{equation}
(u,f)\in\overline{\partial_{t,\rho}\mathcal{M}-\mathcal{M}'+\delta+\mathcal{A}}\label{eq:aux}
\end{equation}
for a suitable $\delta>0.$ 
\begin{lem}
\label{lem:matriallaw-aux}Let $\mathcal{L}\in L(L_{2,\rho}(\R;H))$
and $\delta>0$. Then $\partial_{t,\rho}\mathcal{M}+\mathcal{L}+\delta$
is $(c+\delta-\|\mathcal{L}-\mathcal{N}\|)$-maximal monotone.
\end{lem}

\begin{proof}
For $u\in\dom(\partial_{t,\rho}\mathcal{M})$ we have that
\begin{align*}
\Re\langle(\partial_{t,\rho}\mathcal{M}+\mathcal{L}+\delta)u,u\rangle_{\rho} & =\Re\langle\left(\partial_{t,\rho}\mathcal{M}+\mathcal{N}\right)u,u\rangle_{\rho}+\langle(\delta+(\mathcal{L}-\mathcal{N}))u,u\rangle_{\rho}\\
 & \geq\left(c+\delta-\|\mathcal{L}-\mathcal{N}\|\right)|u|_{\rho}^{2},
\end{align*}
where we have used \prettyref{prop:max_mon_materiallaw}. Moreover,
since 
\[
\partial_{t,\rho}\mathcal{M}+\mathcal{L}+\delta-\left(c+\delta-\|\mathcal{L}-\mathcal{N}\|\right)=\left(\partial_{t,\rho}\mathcal{M}+\mathcal{N}-c\right)+(\|\mathcal{L}-\mathcal{N}\|+\mathcal{L}-\mathcal{N})
\]
the assertion follows by \prettyref{prop:pert}.
\end{proof}
As an immediate consequence we derive the following proposition.
\begin{prop}
\label{prop:c-monotne_aux}Let $\delta>\|\mathcal{M}'+\mathcal{N}\|$.
Then $\partial_{t,\rho}\mathcal{M}-\mathcal{M}'+\delta+\mathcal{A}$
is $c$-monotone and hence, \textup{$\left(\partial_{t,\rho}\mathcal{M}-\mathcal{M}'+\delta+\mathcal{A}\right)^{-1}$
is a Lipschitz-continuous mapping defined on some subset of $L_{2,\rho}(\R;H).$ }
\end{prop}

Thus, in order to prove that \prettyref{eq:aux} is well-posed, it
suffices to prove that the domain of $\left(\partial_{t,\rho}\mathcal{M}-\mathcal{M}'+\delta+\mathcal{A}\right)^{-1}$
is dense in $L_{2,\rho}(\R;H).$ We will prove that for each $f\in H_{\rho}^{1}(\R;H)$
there exists $u\in H_{\rho}^{1}(\R;H)$ such that 
\[
(u,f)\in\partial_{t,\rho}\mathcal{M}-\mathcal{M}'+\delta+\mathcal{A},
\]
which in particular would give 
\[
H_{\rho}^{1}(\R;H)\subseteq\dom\left(\left(\partial_{t,\rho}\mathcal{M}-\mathcal{M}'+\delta+\mathcal{A}\right)^{-1}\right)
\]

and thus, the well-posedness of \prettyref{eq:aux} would follow.
For doing so, let $f\in H_{\rho}^{1}(\R;H)$ and define 
\begin{equation}
u_{\lambda}\coloneqq\left(\partial_{t,\rho}\mathcal{M}-\mathcal{M}'+\delta+\mathcal{A}_{\lambda}\right)^{-1}(f)\quad(\lambda>0).\label{eq:u_lambda}
\end{equation}
Note that $\left(\partial_{t,\rho}\mathcal{M}-\mathcal{M}'+\delta+\mathcal{A}_{\lambda}\right)^{-1}$
with $\delta>\|\mathcal{M}'+\mathcal{N}\|$ is a Lipschitz-continuous
mapping on $L_{2,\rho}(\R;H)$ defined on the whole space by \prettyref{prop:pert}
and \prettyref{lem:matriallaw-aux} and thus $u_{\lambda}\in L_{2,\rho}(\R;H)$
is defined. Our first goal is to prove the following theorem.
\begin{thm}
\label{thm:regular_solution}Let $\delta>\|\mathcal{M}'\|+\|\mathcal{N}\|$
and $f\in H_{\rho}^{1}(\R;H).$ For $\lambda>0$ set 
\[
u_{\lambda}\coloneqq\left(\partial_{t,\rho}\mathcal{M}-\mathcal{M}'+\delta+\mathcal{A}_{\lambda}\right)^{-1}(f).
\]
Then $u_{\lambda}\in H_{\rho}^{1}(\R;H)$ and 
\[
|u_{\lambda}|_{\rho,1}\leq\frac{1}{c}|f|_{\rho,1}.
\]
\end{thm}

In order to prove this theorem, we define the following mapping
\[
\mathcal{B}_{\lambda}\coloneqq\partial_{t,\rho}\mathcal{A}_{\lambda}\partial_{t,\rho}^{-1}:H_{\rho}^{-1}(\R;H)\to H_{\rho}^{-1}(\mathbb{R};H).
\]
Then, by unitary equivalence, $\mathcal{B}_{\lambda}$ is maximal
monotone and Lipschitz-continuous on $H_{\rho}^{-1}(\R;H)$ with $|\mathcal{B}_{\lambda}|_{\mathrm{Lip}}\leq\frac{1}{\lambda}.$
However, we can also interpret $\mathcal{B}_{\lambda}$ as a mapping
on $L_{2,\rho}(\R;H)$ as the following proposition shows.
\begin{prop}
\label{prop:properties_B}Let $\lambda>0$ and define 
\[
\mathcal{B}_{\lambda}\coloneqq\partial_{t,\rho}\mathcal{A}_{\lambda}\partial_{t,\rho}^{-1}:H_{\rho}^{-1}(\R;H)\to H_{\rho}^{-1}(\R;H).
\]
Then $\mathcal{B}_{\lambda}[L_{2,\rho}(\R;H)]\subseteq L_{2,\rho}(\R;H)$
and for $u\in L_{2,\rho}(\R;H)$ we have that 
\begin{align*}
\Re\langle\mathcal{B}_{\lambda}(u),u\rangle_{\rho} & \geq0,\\
|\mathcal{B}_{\lambda}(u)|_{\rho} & \leq\frac{1}{\lambda}|u|_{\rho}.
\end{align*}
\end{prop}

\begin{proof}
Let $u\in L_{2,\rho}(\R;H).$ In order to prove $\mathcal{B}_{\lambda}(u)\in L_{2,\rho}(\R;H)$
we need to show that $\mathcal{A}_{\lambda}(\partial_{t,\rho}^{-1}u)\in H_{\rho}^{1}(\R;H).$
For doing so, we use the notation from \prettyref{lem:difference_quotient}
and estimate for $h>0$
\begin{align*}
|D_{h}\mathcal{A}_{\lambda}(\partial_{t,\rho}^{-1}u)|_{\rho} & =\frac{1}{h}\left|\tau_{h}\mathcal{A}_{\lambda}(\partial_{t,\rho}^{-1}u)-\mathcal{A}_{\lambda}(\partial_{t,\rho}^{-1}u)\right|_{\rho}\\
 & =\frac{1}{h}\left|\mathcal{A}_{\lambda}(\tau_{h}\partial_{t,\rho}^{-1}u)-\mathcal{A}_{\lambda}(\partial_{t,\rho}^{-1}u)\right|_{\rho}\\
 & \leq\frac{1}{\lambda}|D_{h}\partial_{t,\rho}^{-1}u|_{\rho}.
\end{align*}
As $\partial_{t,\rho}^{-1}u\in H_{\rho}^{1}(\R;H)$ the latter estimate
shows $\mathcal{A}_{\lambda}(\partial_{t,\rho}^{-1}u)\in H_{\rho}^{1}(\R;H)$
according to \prettyref{lem:difference_quotient}. Furthermore, letting
$h$ tend to 0, we also infer that 
\[
|\mathcal{B}_{\lambda}(u)|_{\rho}=\lim_{h\to0+}|D_{h}\mathcal{A}_{\lambda}(\partial_{t,\rho}^{-1}u)|_{\rho}\leq\frac{1}{\lambda}\lim_{h\to0+}|D_{h}\partial_{t,\rho}^{-1}u|_{\rho}=\frac{1}{\lambda}|u|_{\rho}.
\]
Moreover, again by using \prettyref{lem:difference_quotient}, we
estimate 
\begin{align*}
\Re\langle\mathcal{B}_{\lambda}(u),u\rangle_{\rho} & =\Re\langle\partial_{t,\rho}\mathcal{A}_{\lambda}(\partial_{t,\rho}^{-1}u),\partial_{t,\rho}\partial_{t,\rho}^{-1}u\rangle_{\rho}\\
 & =\lim_{h\to0+}\frac{1}{h^{2}}\Re\langle\tau_{h}\mathcal{A}_{\lambda}(\partial_{t,\rho}^{-1}u)-\mathcal{A}_{\lambda}(\partial_{t,\rho}^{-1}u),\tau_{h}\partial_{t,\rho}^{-1}u-\partial_{t,\rho}^{-1}u\rangle_{\rho}\\
 & =\lim_{h\to0+}\frac{1}{h^{2}}\Re\langle\mathcal{A}_{\lambda}(\tau_{h}\partial_{t,\rho}^{-1}u)-\mathcal{A}_{\lambda}(\partial_{t,\rho}^{-1}u),\tau_{h}\partial_{t,\rho}^{-1}u-\partial_{t,\rho}^{-1}u\rangle_{\rho}\\
 & \geq0,
\end{align*}
where we have used the monotonicity of $\mathcal{A}_{\lambda}.$ 
\end{proof}
The reason for considering the mapping $\mathcal{B}_{\lambda}$ is
the following. Assume that $u_{\lambda}$ is given by \prettyref{eq:u_lambda}.
Then, at least formally, $\partial_{t,\rho}u$ satisfies 
\[
(\partial_{t,\rho}\mathcal{M}+\delta+\mathcal{B}_{\lambda})(\partial_{t,\rho}u)=\partial_{t,\rho}f.
\]
Thus, we are led to consider the differential equation
\[
(\partial_{t,\rho}\mathcal{M}+\delta+\mathcal{B}_{\lambda})(v)=g.
\]
Since $\mathcal{B}_{\lambda}$ is Lipschitz-continuous and maximal
monotone on $H_{\rho}^{-1}(\mathbb{R};H)$, the latter equation is
well-posed in $H_{\rho}^{-1}(\R;H)$, if we can show that $\partial_{t,\rho}\mathcal{M}+\delta$
is $c$-maximal monotone on $H_{\rho}^{-1}(\R;H)$ in some sense.
This will be shown in the next proposition.
\begin{prop}
\label{prop:T_delta}Let $\delta>\|\mathcal{M}'\|+\|\mathcal{N}\|.$
Then the mapping 
\[
\partial_{t,\rho}\mathcal{M}+\delta:L_{2,\rho}(\R;H)\subseteq H_{\rho}^{-1}(\R;H)\to H_{\rho}^{-1}(\R;H)
\]
is closable and its closure, denoted by $\mathcal{T}_{\delta}$, is
$(c+\delta-\|\mathcal{M}'+\mathcal{N}\|)$-maximal monotone.
\end{prop}

\begin{proof}
To prove the closability, let $(v_{n})_{n\in\mathbb{N}}$ be a sequence
in $L_{2,\rho}(\R;H)$ such that $v_{n}\to0$ and $\left(\partial_{t,\rho}\mathcal{M}+\delta\right)v_{n}\to y$
for some $y\in H_{\rho}^{-1}(\R;H),$ where both convergences are
with respect to the topology in $H_{\rho}^{-1}(\R;H).$ Since clearly
$\delta v_{n}\to0$ in $H_{\rho}^{-1}(\R;H),$ we infer that $\partial_{t,\rho}\mathcal{M}v_{n}\to y$
in $H_{\rho}^{-1}(\R;H).$ Moreover, noting that $\partial_{t,\rho}^{-1}v_{n}\to0$
in $L_{2,\rho}(\R;H)$ we derive
\[
\mathcal{M}v_{n}=\mathcal{M}\partial_{t,\rho}\partial_{t,\rho}^{-1}v_{n}=\partial_{t,\rho}\mathcal{M}\partial_{t,\rho}^{-1}v_{n}-\mathcal{M}'\partial_{t,\rho}^{-1}v_{n}\to0
\]
in $H_{\rho}^{-1}(\R;H)$ and hence, $\partial_{t,\rho}\mathcal{M}v_{n}\to0$
in $H_{\rho}^{-2}(\R;H).$ This gives $y=0$ and thus, the operator
is closable.

Let now $v\in L_{2,\rho}(\R;H).$ Then we estimate
\begin{align*}
\Re\left\langle \left(\partial_{t,\rho}\mathcal{M}+\delta\right)v,v\right\rangle _{\rho,-1} & =\Re\langle\mathcal{M}v+\partial_{t,\rho}^{-1}\delta v,\partial_{t,\rho}^{-1}v\rangle_{\rho}\\
 & =\Re\langle\mathcal{M}\partial_{t,\rho}\partial_{t,\rho}^{-1}v+\partial_{t,\rho}^{-1}\delta v,\partial_{t,\rho}^{-1}v\rangle_{\rho}\\
 & =\Re\langle\left(\partial_{t,\rho}\mathcal{M}-\mathcal{M}'+\delta\right)\partial_{t,\rho}^{-1}v,\partial_{t,\rho}^{-1}v\rangle_{\rho}\\
 & \geq(c+\delta-\|\mathcal{M}'+\mathcal{N}\|)|\partial_{t,\rho}^{-1}v|_{\rho}^{2}\\
 & =(c+\delta-\|\mathcal{M}'+\mathcal{N}\|)|v|_{\rho,-1}^{2},
\end{align*}
where we have used \prettyref{lem:matriallaw-aux}. Since $L_{2,\rho}(\R;H)$
is a core for $\mathcal{T_{\delta}}$, the $(c+\delta-\|\mathcal{M}'+\mathcal{N}\|)$-monotonicity
follows.

For proving the maximal monotonicity, it suffices to show that $\ran(\mathcal{T}_{\delta})$
is dense in $H_{\rho}^{-1}(\R;H)$. Note that, since $\delta>\|\mathcal{N}\|$,
we have that $\partial_{t,\rho}\mathcal{M}+\delta$ is $c$-maximal
monotone on $L_{2,\rho}(\R;H)$ by \prettyref{lem:matriallaw-aux}.
Thus, in particular $L_{2,\rho}(\R;H)\subseteq\ran(\mathcal{T}_{\delta})$
and hence, the assertion follows. 
\end{proof}
\begin{defn*}
We define the set 
\[
K\coloneqq\{\delta>\|\mathcal{M}'\|+\|\mathcal{N}\|\,;\,\forall g\in L_{2,\rho}(\R;H):\,(\mathcal{T}_{\delta}+\mathcal{B}_{\lambda})^{-1}(g)\in L_{2,\rho}(\R;H)\}.
\]
\end{defn*}
We aim to prove that $K=]\|\mathcal{M}'\|+\|\mathcal{N}\|,\infty[.$
We start with the following observation.
\begin{lem}
\label{lem:uniform_bound_K}Let $\delta\in K.$ Then 
\[
\left|\left(\mathcal{T_{\delta}}+\mathcal{B}_{\lambda}\right)^{-1}(g)\right|_{\rho}\leq\frac{1}{c}|g|_{\rho}
\]
for each $g\in L_{2,\rho}(\R;H).$
\end{lem}

\begin{proof}
We set $v\coloneqq\left(\mathcal{T_{\delta}}+\mathcal{B}_{\lambda}\right)^{-1}(g)\in L_{2,\rho}(\R;H).$
Then $\left(\partial_{t,\rho}\mathcal{M}+\delta\right)v=\mathcal{T_{\delta}}v=g-B_{\lambda}(v)\in L_{2,\rho}(\R;H)$
by \prettyref{prop:properties_B} and hence, $v\in\dom(\partial_{t,\rho}\mathcal{M}).$
We then estimate by using \prettyref{prop:properties_B}
\begin{align*}
\Re\langle g,v\rangle_{\rho} & =\Re\langle\mathcal{T}_{\delta}v+B_{\lambda}(v),v\rangle_{\rho}\\
 & \geq\Re\langle(\partial_{t,\rho}\mathcal{M}+\delta)v,v\rangle_{\rho}\\
 & \geq c|v|_{\rho}^{2},
\end{align*}
where we again have used that $\partial_{t,\rho}\mathcal{M}+\delta$
is $c$-maximal monotone on $L_{2,\rho}(\R;H)$ since $\delta>\|\mathcal{N}\|.$
By applying the Cauchy-Schwarz inequality on the left hand side, we
derive the desired inequality. 
\end{proof}
\begin{cor}
\label{cor:nonempty_suffices}If $K\ne\emptyset$ then $K=]\|\mathcal{M}'\|+\|\mathcal{N}\|,\infty[.$
\end{cor}

\begin{proof}
We prove that for $\delta\in K$ it follows that $B(\delta,c)\cap]\|\mathcal{M}'\|+\|\mathcal{N}\|,\infty[\subseteq K.$
This would yield the assertion. So let $\delta\in K$ and choose $\delta'>\|\mathcal{M}'\|+\|\mathcal{N}\|$
with $|\delta-\delta'|<c.$ For $g\in L_{2,\rho}(\R;H)$ we define
the sequence $(v_{n})_{n\in\mathbb{N}}$ in $L_{2,\rho}(\R;H)$ recursively
by $v_{0}\coloneqq0$ and 
\[
v_{n+1}\coloneqq\left(\mathcal{T}_{\delta}+\mathcal{B}_{\lambda}\right)^{-1}(g-(\delta'-\delta)v_{n})
\]
for $n\in\N.$ Note that indeed $v_{n+1}\in L_{2,\rho}(\R;H)$ since
$\delta\in K.$ Moreover, since $|(\mathcal{T}_{\delta}+\mathcal{B}_{\lambda})^{-1}|_{\mathrm{Lip},H_{\rho}^{-1}(\R;H)}\leq\frac{1}{c}$
, we infer that the mapping 
\[
H_{\rho}^{-1}(\R;H)\ni w\mapsto\left(\mathcal{T}_{\delta}+\mathcal{B}_{\lambda}\right)^{-1}\left(g-(\delta'-\delta)w\right)\in H_{\rho}^{-1}(\R;H)
\]
is a strict contraction and thus, it has a unique fixed point $v\in H_{\rho}^{-1}(\R;H).$
Moreover, $v_{n}\to v$ in $H_{\rho}^{-1}(\R;H)$ as $n\to\infty.$
Note, that the fixed point $v$ satisfies 
\[
(\mathcal{T}_{\delta'}+\mathcal{B}_{\lambda})(v)=g
\]
and thus, to complete the proof, we need to show that $v\in L_{2,\rho}(\R;H).$
For doing so, it suffices to show that $\sup_{n\in\mathbb{N}}|v_{n}|_{\rho}<\infty$
by \prettyref{lem:better_limit}. Using \prettyref{lem:uniform_bound_K},
we estimate 
\[
|v_{n+1}|_{\rho}\leq\frac{1}{c}(|g|_{\rho}+|\delta'-\delta||v_{n}|_{\rho})
\]
for $n\in\N$ and thus, by induction 
\[
|v_{n}|_{\rho}\leq\frac{1}{c}|g|_{\rho}\sum_{j=0}^{n-1}\left(\frac{|\delta'-\delta|}{c}\right)^{j}\quad(n\in\N).
\]
As $|\delta'-\delta|<c,$ we infer the boundedness of $(v_{n})_{n\in\N}$
and hence, the assertion follows. 
\end{proof}
\begin{prop}
\label{prop:K is all}$K=]\|\mathcal{M}'\|+\|\mathcal{N}\|,\infty[.$
\end{prop}

\begin{proof}
We choose $\delta>\frac{1}{\lambda}+\|\mathcal{N}\|+\|\mathcal{M}'\|$
and prove that $\delta\in K.$ Then the assertion follows by \prettyref{cor:nonempty_suffices}.
Let $g\in L_{2,\rho}(\R;H)$ and define the sequence $(v_{n})_{n\in\N}$
in $L_{2,\rho}(\R;H)$ recursively by $v_{0}\coloneqq0$ and 
\[
v_{n+1}\coloneqq\mathcal{T}_{\delta}^{-1}(g-B_{\lambda}(v_{n}))=(\partial_{t,\rho}\mathcal{M}+\delta)^{-1}(g-\mathcal{B}_{\lambda}(v_{n}))\quad(n\in\N).
\]
Note that $\mathcal{B}_{\lambda}$ leaves $L_{2,\rho}(\R;H)$ invariant
by \prettyref{prop:properties_B} and that $\partial_{t,\rho}\mathcal{M}+\delta$
is $c$-maximal monotone on $L_{2,\rho}(\R;H)$ and thus, $v_{n+1}$
lies indeed in $L_{2,\rho}(\R;H).$ Indeed, $\partial_{t,\rho}\mathcal{M}+\delta$
is even $\left(c+\delta-\|\mathcal{N}\|\right)$-maximal monotone
and thus, we estimate 
\[
|v_{n+1}|_{\rho}\leq\frac{1}{c+\delta-\|\mathcal{N}\|}\left(|g|_{\rho}+|\mathcal{B}_{\lambda}(v_{n})|_{\rho}\right)\leq\frac{1}{c+\delta-\|\mathcal{N}\|}\left(|g|_{\rho}+\frac{1}{\lambda}|(v_{n})|_{\rho}\right)\quad(n\in\N),
\]
where we have used \prettyref{prop:properties_B} in the second inequality.
Thus, by induction we can estimate 
\[
|v_{n}|_{\rho}\leq\frac{1}{c+\delta-\|\mathcal{N}\|}|g|_{\rho}\sum_{j=0}^{n-1}\frac{1}{\lambda^{j}\left(c+\delta-\|\mathcal{N}\|\right)^{j}}\quad(n\in\N)
\]
and thus, by the choice of $\delta,$ the sequence $(v_{n})_{n\in\N}$
is bounded in $L_{2,\rho}(\R;H).$ Moreover, since $\mathcal{T}_{\delta}$
is $(c+\delta-\|\mathcal{M}'+\mathcal{N}\|)$-maximal monotone by
\prettyref{prop:T_delta}, we infer by the choice of $\delta$ that
$|\mathcal{T}_{\delta}|_{\mathrm{Lip},H_{\rho}^{-1}(\R;H)}<\lambda$
and hence 
\[
H_{\rho}^{-1}(\R;H)\ni w\mapsto\mathcal{T}_{\delta}^{-1}(g-\mathcal{B}_{\lambda}(w))\in H_{\rho}^{-1}(\R;H)
\]
is a strict contraction and hence, has a unique fixed point $v\in H_{\rho}^{-1}(\R;H)$
which satisfies 
\[
(\mathcal{T}_{\delta}+\mathcal{B}_{\lambda})(v)=g.
\]
Since $v_{n}\to v$ in $H_{\rho}^{-1}(\R;H)$, we infer that $v\in L_{2,\rho}(\R;H)$
by \prettyref{lem:better_limit} and thus, $\delta\in K$. 
\end{proof}
We now can come to the
\begin{proof}[Proof of \prettyref{thm:regular_solution}]
 Let $f\in H_{\rho}^{1}(\R;H)$ and set 
\[
u_{\lambda}\coloneqq\left(\partial_{t,\rho}\mathcal{M}-\mathcal{M}'+\delta+\mathcal{A}_{\lambda}\right)^{-1}(f).
\]
Moreover, set 
\[
v\coloneqq\left(\mathcal{T}_{\delta}+\mathcal{B}_{\lambda}\right)^{-1}(\partial_{t,\rho}f).
\]
We note that $v\in L_{2,\rho}(\R;H)$ by \prettyref{prop:K is all}
and that 
\[
|v|_{\rho}\leq\frac{1}{c}|\partial_{t,\rho}f|_{\rho}=\frac{1}{c}|f|_{\rho,1}
\]
by \prettyref{lem:uniform_bound_K}. We compute 
\begin{align*}
\partial_{t,\rho}\mathcal{M}\partial_{t,\rho}^{-1}v-\mathcal{M}'\partial_{t,\rho}^{-1}v+\delta\partial_{t,\rho}^{-1}v & =\mathcal{M}v+\delta\partial_{t,\rho}^{-1}v\\
 & =\partial_{t,\rho}^{-1}\mathcal{T}_{\delta}v\\
 & =\partial_{t,\rho}^{-1}(\partial_{t,\rho}f-\mathcal{B}_{\lambda}(v))\\
 & =f-\mathcal{A}_{\lambda}(\partial_{t,\rho}^{-1}v)
\end{align*}
and hence, 
\[
\left(\partial_{t,\rho}\mathcal{M}-\mathcal{M}'+\delta+\mathcal{A}_{\lambda}\right)(\partial_{t,\rho}^{-1}v)=f,
\]
which shows $u_{\lambda}=\partial_{t,\rho}^{-1}v\in H_{\rho}^{1}(\R;H).$
Moreover, 
\[
|u_{\lambda}|_{\rho,1}=|v|_{\rho}\leq\frac{1}{c}|f|_{\rho,1}.\tag*{\qedhere}
\]
\end{proof}
\begin{cor}
\label{cor:maxmonaux}Let $\delta>\|\mathcal{M}'\|+\|\mathcal{N}\|$.
Then $\overline{\partial_{t,\rho}\mathcal{M}-\mathcal{M}'+\delta+\mathcal{A}}$
is $c$-maximal monotone and hence, 
\[
\mathcal{S}_{\rho,\mathrm{aux}}\coloneqq\left(\overline{\partial_{t,\rho}\mathcal{M}-\mathcal{M}'+\delta+\mathcal{A}}\right)^{-1}:L_{2,\rho}(\R;H)\to L_{2,\rho}(\R;H)
\]
is Lipschitz-continuous. Moreover, for $f\in H_{\rho}^{1}(\R;H)$
we have that $\mathcal{S}_{\rho,\mathrm{aux}}(f)\in H_{\rho}^{1}(\R;H)\cap\dom\left(\left(\partial_{t,\rho}\mathcal{M}-\mathcal{M}'+\delta+\mathcal{A}\right)^{-1}\right).$
\end{cor}

\begin{proof}
We already know that $\overline{\partial_{t,\rho}\mathcal{M}-\mathcal{M}'+\delta+\mathcal{A}}$
is $c$-monotone by \prettyref{prop:c-monotne_aux}. For proving the
maximal monotonicity it suffices to show that there is a dense subset
$D\subseteq L_{2,\rho}(\R;H)$ such that for each $f\in D$ there
exists $u\in L_{2,\rho}(\R;H)$ with 
\[
(u,f)\in\overline{\partial_{t,\rho}\mathcal{M}-\mathcal{M}'+\delta+\mathcal{A}}.
\]
We show that the latter is true for $D=H_{\rho}^{1}(\R;H).$ So, let
$f\in H_{\rho}^{1}(\R;H).$ For showing that a solution $u\in L_{2,\rho}(\R;H)$
exists, it suffices to prove that $\sup_{\lambda>0}|\mathcal{A}_{\lambda}(u_{\lambda})|_{\rho}<\infty$
for 
\[
u_{\lambda}\coloneqq\left(\partial_{t,\rho}\mathcal{M}-\mathcal{M}'+\delta+\mathcal{A}_{\lambda}\right)^{-1}(f)
\]
by \prettyref{prop:pert-yosida}. According to \prettyref{thm:regular_solution}
we know that $u_{\lambda}\in H_{\rho}^{1}(\R;H)$ and that 
\[
|u_{\lambda}|_{\rho,1}\leq\frac{1}{c}|f|_{\rho,1}
\]
for each $\lambda>0.$ Thus, we can estimate 
\begin{align*}
\left|\mathcal{A}_{\lambda}(u_{\lambda})\right|_{\rho} & =\left|f-\left(\partial_{t,\rho}\mathcal{M}-\mathcal{M}'\right)u_{\lambda}-\delta u_{\lambda}\right|_{\rho}\\
 & =\left|f-\mathcal{M}\partial_{t,\rho}u_{\lambda}-\delta u_{\lambda}\right|_{\rho}\\
 & \leq|f|_{\rho}+\|\mathcal{M}\||u_{\lambda}|_{\rho,1}+\frac{\delta}{\rho}|u_{\lambda}|_{\rho,1}\\
 & \leq|f|_{\rho}+\frac{1}{c}\left(\|\mathcal{M}\|+\frac{\delta}{\rho}\right)|f|_{\rho,1}.
\end{align*}
Summarising, we have shown the $c$-maximal monotonicity of $\overline{\partial_{t,\rho}\mathcal{M}-\mathcal{M}'+\delta+\mathcal{A}}$.
Moreover, $u_{\lambda}\to u$ as $\lambda\to0+$ in $L_{2,\rho}(\R;H)$
and 
\[
(u,f)\in\partial_{t,\rho}\mathcal{M}-\mathcal{M}'+\delta+\mathcal{A}
\]
by \prettyref{prop:pert-yosida}, i.e. $u=\mathcal{S}_{\rho,\mathrm{aux}}(f)\in\dom\left(\left(\partial_{t,\rho}\mathcal{M}-\mathcal{M}'+\delta+\mathcal{A}\right)^{-1}\right)$.
Moreover, since $(u_{\lambda})_{\lambda>0}$ is bounded in $H_{\rho}^{1}(\R;H),$
we infer that also $u=\mathcal{S}_{\rho,\mathrm{aux}}(f)\in H_{\rho}^{1}(\R;H)$
by \prettyref{lem:better_limit}.
\end{proof}
With this result at hand, we can prove our main theorem.
\begin{proof}[Proof of \prettyref{thm:main}]
 We have that $\overline{\partial_{t,\rho}\mathcal{M}+\mathcal{N}+\mathcal{A}}$
is $c$-monotone by \prettyref{prop:max_mon_materiallaw} and hence,
\[
\overline{\partial_{t,\rho}\mathcal{M}+\mathcal{N}+\mathcal{A}}-c=\left(\overline{\partial_{t,\rho}\mathcal{M}-\mathcal{M}'+\delta+\mathcal{A}}-c\right)+(\mathcal{M}'+\mathcal{N}-\delta)
\]
is maximal monotone by \prettyref{prop:pert} and \prettyref{cor:maxmonaux}.
\end{proof}

\section{Causality and independence of $\rho$}

The aim of this section is twofold. First, we prove that the solution
operator $\mathcal{S}_{\rho}$ associated with the differential inclusion
\[
(u,f)\in\overline{\partial_{t,\rho}\mathcal{M}+\mathcal{N}+\mathcal{A}}
\]
is causal under suitable additional conditions on $\mathcal{M},\mathcal{N}$
and $\mathcal{A}.$ And second, we prove that the solution operator
is independent of the choice of the parameter $\rho$, if $\mathcal{M},\mathcal{N}$
and $\mathcal{A}$ are independent of $\rho$ in some sense.

We start with the causality and define, what we mean by a causal operator. 
\begin{defn*}
Let $F:L_{2,\rho}(\R;H)\to L_{2,\rho}(\R;H).$ Then, $F$ is called
\emph{causal}, if for each $f,g\in L_{2,\rho}(\R;H)$ when $f=g$
on $]-\infty,a]$ for some $a\in\mathbb{R}$, also $F(f)=F(g)$ on
$]-\infty,a]$ holds. 
\end{defn*}
\begin{rem}
\begin{enumerate}[(a)]

\item Causality is a crucial property for operators describing a
temporal evolution, as it says that the solution $u$ up to some time
$a\in\R$ does not depend on the behaviour of the right hand side
$f$ on $[a,\infty[.$ 

\item For mappings not defined on the whole space $L_{2,\rho}(\R;H)$
one has to adopt the definition of causality in a way, that it is
preserved under closure (i.e. $\overline{F}$ is causal if $F$ is
causal) and that it coincides with the causality defined above if
$F$ is defined on the whole $L_{2,\rho}(\R;H)$. This was done in
\cite{Waurick2013_causality}.

\end{enumerate}
\end{rem}

We now state the hypotheses which allow us to deduce the causality
of the solution operator.

\begin{hyp} \label{hyp:causality}Let $\mathcal{M},\mathcal{N},\mathcal{M}'\in L(L_{2,\rho}(\R;H))$
be such that:

\begin{enumerate}[(a)]

\item $\mathcal{M}\partial_{t,\rho}\subseteq\partial_{t,\rho}\mathcal{M}-\mathcal{M}',$

\item There exists $c>0$ such that 
\[
\Re\langle(\partial_{t,\rho}\mathcal{M}+\mathcal{N})\varphi,\chi_{\mathbb{R}_{\leq a}}\varphi\rangle_{\rho}\geq c|\chi_{\mathbb{R}_{\leq a}}\varphi|_{\rho}^{2}
\]
for each $\varphi\in C_{c}^{\infty}(\R;H)$ and each $a\in\R.$

\end{enumerate}

Moreover, let $\mathcal{A}\subseteq L_{2,\rho}(\R;H)\times L_{2,\rho}(\R;H)$
be maximal monotone such that 
\[
\forall u,v\in L_{2,\rho}(\R;H):\,\left((u,v)\in\mathcal{A}\Rightarrow\forall h\geq0:(\tau_{h}u,\tau_{h}v)\in\mathcal{A}\right)
\]
and assume that $\chi_{\mathbb{R}_{\leq a}}\mathcal{A}$ is monotone
for each $a\in\R.$

\end{hyp}
\begin{rem}
We note that \prettyref{hyp:causality} imply \prettyref{hyp:assumptions}
since (b) of \prettyref{hyp:assumptions} follows from \prettyref{hyp:causality}
(b) by letting $a$ tend to infinity. Moreover, the monotonicity of
$\chi_{\mathbb{R}_{\leq a}}\mathcal{A}$ states that for each $(u,v),(x,y)\in\mathcal{A}$
we have that 
\[
\Re\langle u-x,\chi_{\mathbb{R}_{\leq a}}(v-y)\rangle_{\rho}=\Re\langle\chi_{\mathbb{R}_{\leq a}}(u-x),v-y\rangle_{\rho}\geq0.
\]
\end{rem}

\begin{lem}
\label{lem:causality_better}Assume \prettyref{hyp:causality}. Then
\[
\Re\langle(\partial_{t,\rho}\mathcal{M}+\mathcal{N})u,\chi_{\mathbb{R}_{\leq a}}u\rangle_{\rho}\geq c|\chi_{\mathbb{R}_{\leq a}}u|_{\rho}^{2}
\]
for each $u\in\dom(\partial_{t,\rho}\mathcal{M})$ and $a\in\R.$
\end{lem}

\begin{proof}
By \prettyref{prop:max_mon_materiallaw}, $C_{c}^{\infty}(\R;H)$
is a core for $\partial_{t,\rho}\mathcal{M}.$ Hence, the assertion
follows by approximating $u$ with elements in $C_{c}^{\infty}(\R;H)$
and noting that multiplication with $\chi_{\R_{\leq a}}$ is continuous
on $L_{2,\rho}(\R;H).$ 
\end{proof}
\begin{thm}[Causality]
\label{thm:causaliyt} Assume \prettyref{hyp:causality}. Then the
solution operator $\mathcal{S}_{\rho}\coloneqq\left(\overline{\partial_{t,\rho}\mathcal{M}+\mathcal{N}+\mathcal{A}}\right)^{-1}$
from \prettyref{thm:main} is causal.
\end{thm}

\begin{proof}
Let $f,g\in L_{2,\rho}(\R;H)$ and assume that $f=g$ on $]-\infty,a]$
for some $a\in\R.$ We set 
\[
u\coloneqq\mathcal{S}_{\rho}(f)\text{ and }v\coloneqq\mathcal{S}_{\rho}(g).
\]
Since $\mathcal{S}_{\rho}=\left(\overline{\partial_{t,\rho}\mathcal{M}+\mathcal{N}+\mathcal{A}}\right)^{-1}=\overline{\left(\partial_{t,\rho}\mathcal{M}+\mathcal{N}+\mathcal{A}\right)^{-1}}$
is defined on the whole $L_{2,\rho}(\R;H)$ it follows that $\dom\left((\partial_{t,\rho}\mathcal{M}+\mathcal{N}+\mathcal{A})^{-1}\right)$
is dense in $L_{2,\rho}(\R;H).$ Hence, there exist sequences $(f_{n})_{n\in\N}$
and $(g_{n})_{n\in\N}$ in $\dom\left((\partial_{t,\rho}\mathcal{M}+\mathcal{N}+\mathcal{A})^{-1}\right)$
such that $f_{n}\to f$ and $g_{n}\to g$ in $L_{2,\rho}(\R;H).$
Setting 
\[
u_{n}\coloneqq\mathcal{S}_{\rho}(f_{n})\text{ and }v_{n}\coloneqq\mathcal{S}_{\rho}(g_{n})
\]
for $n\in\N$, it follows that $u_{n}\to u$ and $v_{n}\to v$ in
$L_{2,\rho}(\R;H)$ by continuity of $\mathcal{S}_{\rho}.$ By definition,
for $n\in\N$ there exist $x_{n},y_{n}\in L_{2,\rho}(\R;H)$ such
that $(u_{n},x_{n}),(v_{n},y_{n})\in\mathcal{A}$ and 
\begin{align*}
\partial_{t,\rho}\mathcal{M}u_{n}+\mathcal{N}u_{n}+x_{n} & =f_{n}\\
\partial_{t,\rho}\mathcal{M}v_{n}+\mathcal{N}v_{n}+y_{n} & =g_{n}.
\end{align*}
Moreover, by \prettyref{lem:causality_better} and the monotonicity
of $\chi_{\R_{\leq a}}\mathcal{A}$ we estimate 
\begin{align*}
 & \Re\langle f_{n}-g_{n},\chi_{\R_{\leq a}}(u_{n}-v_{n})\rangle_{\rho}\\
 & =\Re\langle\left(\partial_{t,\rho}\mathcal{M}+\mathcal{N}\right)(u_{n}-v_{n}),\chi_{\mathbb{R}_{\leq a}}(u_{n}-v_{n})\rangle_{\rho}+\Re\langle x_{n}-y_{n},\chi_{\mathbb{R}_{\leq a}}(u_{n}-v_{n})\rangle_{\rho}\\
 & \geq c|\chi_{\mathbb{R}_{\leq a}}(u_{n}-v_{n})|_{\rho}^{2}.
\end{align*}
Letting $n$ tend to infinity, it follows that 
\[
0=\Re\langle f-g,\chi_{\R_{\leq a}}(u-v)\rangle_{\rho}\geq c|\chi_{\mathbb{R}_{\leq a}}(u-v)|_{\rho}^{2},
\]
which proves $u=v$ on $]-\infty,a]$ and hence, the causality of
$\mathcal{S}_{\rho}$ follows. 
\end{proof}
In order to formulate the independence result we need the concept
of evolutionary mappings. These mappings were introduced in \cite{Wauricl2014,Waurick_habil},
however in a more general way than needed here. 
\begin{defn*}
Let $\rho_{0}>0$. A linear mapping $T:C_{c}^{\infty}(\R;H)\to\bigcap_{\mu\geq\rho_{0}}L_{2,\mu}(\R;H)$
is called \emph{evolutionary at $\rho_{0}$, }if for each $\rho\geq\rho_{0}$
the mapping 
\[
T:C_{c}^{\infty}(\R;H)\subseteq L_{2,\rho}(\R;H)\to L_{2,\rho}(\R;H)
\]
is bounded and $\|T\|_{\mathrm{ev},\rho_{0}}\coloneqq\sup_{\rho\geq\rho_{0}}\|T\|_{L(L_{2,\rho}(\R;H))}<\infty$.
If $T$ is evolutionary at $\rho_{0}$, we denote the (unique) extension
of $T$ to $L_{2,\rho}(\R;H)$ by $T_{\rho}$ for each $\rho\geq\rho_{0}.$ 
\end{defn*}
We are now ready to formulate our hypotheses, which will allow us
to prove the independence of the parameter $\rho.$ 

\begin{hyp}\label{hyp:independence} Let $\rho_{0}>0$ and $\mathcal{M},\mathcal{N},\mathcal{M}'$
evolutionary at $\rho_{0}$ such that

\begin{enumerate}[(a)]

\item $\forall\varphi\in C_{c}^{\infty}(\R;H):\,\mathcal{M}\varphi'+\mathcal{M}'\varphi=(\mathcal{M}\varphi)'$
in the sense of distributions.

\item There exists $c>0$ such that for all $\rho\geq\rho_{0}$ 
\[
\Re\langle(\partial_{t,\rho}\mathcal{M}+\mathcal{N})\varphi,\chi_{\mathbb{R}_{\leq a}}\varphi\rangle_{\rho}\geq c|\chi_{\mathbb{R}_{\leq a}}\varphi|_{\rho}^{2}
\]
for each $\varphi\in C_{c}^{\infty}(\R;H)$ and each $a\in\R.$

\end{enumerate}

Moreover, let $A\subseteq H\times H$ be maximal monotone with $(0,0)\in A.$

\end{hyp}
\begin{rem}
We note that by \prettyref{lem:lifting} the extension of $A$ to
$L_{2,\rho}(\R;H)$ is maximal monotone for each $\rho\in\R.$ We
will denote this extension by $\mathcal{A}_{\rho}.$
\end{rem}

First, we show that \prettyref{hyp:independence} imply \prettyref{hyp:causality}
for the operators $\mathcal{M}_{\rho},\mathcal{N}_{\rho},\mathcal{M}_{\rho}'$
and the relation $\mathcal{A}_{\rho}$ for each $\rho\geq\rho_{0}.$
\begin{lem}
\label{lem:evo_implies_causal}We assume \prettyref{hyp:independence}.
Then for each $\rho\geq\rho_{0}$ the operators $\mathcal{M}_{\rho},\mathcal{N}_{\rho},\mathcal{M}_{\rho}'$
and the relation $\mathcal{A}_{\rho}$ satisfy \prettyref{hyp:causality}.
Moreover, the constant $c$ in \prettyref{hyp:causality} (b) is independent
of $\rho$. 
\end{lem}

\begin{proof}
It is clear that $\mathcal{M}_{\rho},\mathcal{N}_{\rho},\mathcal{M}_{\rho}'$
satisfy \prettyref{hyp:causality} (b). Moreover, the constant $c$
in \prettyref{hyp:causality} (b) is the same as in \prettyref{hyp:independence}
(b) and thus, independent of $\rho.$ To show \prettyref{hyp:causality}
(a), let $u\in\dom(\partial_{t,\rho}).$ We need to show that $\mathcal{M}_{\rho}u\in\dom(\partial_{t,\rho})$
and 
\[
\partial_{t,\rho}\mathcal{M}_{\rho}u=\mathcal{M}_{\rho}\partial_{t,\rho}u+\mathcal{M}_{\rho}'u.
\]
Choose a sequence $(\varphi_{n})_{n\in\N}$ in $C_{c}^{\infty}(\R;H)$
such that $\varphi_{n}\to u$ in $H_{\rho}^{1}(\R;H).$ Note that
by \prettyref{hyp:independence} (a) we have that 
\[
\left(\mathcal{M}\varphi_{n}\right)'=\mathcal{M}\varphi'_{n}+\mathcal{M}'\varphi_{n}\in L_{2,\rho}(\R;H)\quad(n\in\N)
\]
and hence, $\mathcal{M}\varphi_{n}\in\dom(\partial_{t,\rho})$ (compare
\prettyref{rem:derivative} (a)). Moreover, 
\[
\partial_{t,\rho}\mathcal{M}\varphi_{n}=\mathcal{M}\varphi'_{n}+\mathcal{M}'\varphi_{n}\to\mathcal{M}_{\rho}\partial_{t,\rho}u+\mathcal{M}'_{\rho}u\quad(n\to\infty)
\]
in $L_{2,\rho}(\R;H)$ and since also $\mathcal{M}\varphi_{n}\to\mathcal{M}_{\rho}u$,
the assertion follows by the closedness of $\partial_{t,\rho}.$ It
is left to prove that $\mathcal{A}_{\rho}$ satisfies \prettyref{hyp:independence}.
We had already remarked that $\mathcal{A}_{\rho}$ is maximal monotone
by \prettyref{lem:lifting}. Moreover, for $(u,v)\in\mathcal{A}_{\rho}$
and $h\geq0$ we have that 
\begin{align*}
(u,v)\in\mathcal{A}_{\rho} & \Rightarrow(u(t),v(t))\in A\quad(t\in\R\text{ a.e.})\\
 & \Rightarrow(u(t+h),v(t+h))\in A\quad(t\in\R\text{ a.e.})\\
 & \Rightarrow(\tau_{h}u,\tau_{h}v)\in\mathcal{A}_{\rho}.
\end{align*}
Moreover, for $a\in\R$ and $(u,v),(x,y)\in\mathcal{A}_{\rho}$ we
have that 
\[
\Re\langle u-x,\chi_{\R_{\leq a}}(v-y)\rangle_{\rho}=\intop_{-\infty}^{a}\Re\langle u(t)-x(t),v(t)-y(t)\rangle\e^{-2\rho t}\text{ d}t\geq0,
\]
since $(u(t),v(t)),(x(t),y(t))\in A$ for almost every $t\in\R$ and
$A$ is monotone. 
\end{proof}
In order to prove that the solution operator $\mathcal{S}_{\rho}=\left(\overline{\partial_{t,\rho}\mathcal{M}_{\rho}+\mathcal{N}_{\rho}+\mathcal{A}_{\rho}}\right)^{-1}$
is independent of $\rho$ in the sense that for each $\mu,\rho\geq\rho_{0}$
\[
\mathcal{S}_{\rho}(f)=\mathcal{S}_{\mu}(f)\quad(f\in L_{2,\rho}(\R;H)\cap L_{2,\mu}(\R;H))
\]
we prove this property for the solution operator $\mathcal{S}_{\rho,\mathrm{aux}}=\left(\overline{\partial_{t,\rho}\mathcal{M}_{\rho}-\mathcal{M}_{\rho}'+\delta+\mathcal{A}_{\rho}}\right)^{-1}$
with $\delta>\|\mathcal{M}'\|_{\mathrm{ev},\rho_{0}}+\|\mathcal{N}\|_{\mathrm{ev},\rho_{0}}$
first. For doing so, we need to show that the solution operator $\mathcal{S}_{\rho,\mathrm{aux}}$
is causal. This will be shown in the next proposition. However, we
first need the following lemma, showing an important property of evolutionary
mappings.
\begin{lem}
\label{lem:causality_evo}Let $\rho_{0}>0$ and $T:C_{c}^{\infty}(\R;H)\to\bigcap_{\mu\geq\rho_{0}}L_{2,\mu}(\R;H)$
evolutionary at $\rho_{0}.$ Then, for each $\mu,\rho\geq\rho_{0}$
and $f\in L_{2,\rho}(\R;H)\cap L_{2,\mu}(\R;H)$ we have that 
\[
T_{\mu}f=T_{\rho}f.
\]
Moreover, for each $\rho\geq\rho_{0}$ the mapping $T_{\rho}$ is
causal.
\end{lem}

\begin{proof}
The first assertion is clear, since a function $f\in L_{2,\rho}(\R;H)\cap L_{2,\mu}(\R;H)$
can be approximated by one sequence $(\varphi_{n})_{n\in\N}$ in $C_{c}^{\infty}(\R;H)$
such that $\varphi_{n}\to f$ in $L_{2,\mu}(\R;H)$ and $L_{2,\rho}(\R;H)$,
see e.g. \cite[Lemma 3.5]{Trostorff2013_stability}. For proving the
causality of $T_{\rho}$ we note that due to linearity it suffices
to show for $u\in L_{2,\rho}(\R;H)$ with $\spt u\subseteq\R_{\geq a}$
for some $a\in\R$ it follows that $\spt T_{\rho}u\subseteq\R_{\geq a}.$
So, let $u\in L_{2,\rho}(\R;H)$ with $\spt u\subseteq\R_{\geq a}$
for some $a\in\R$. Then $u\in\bigcap_{\mu\geq\rho}L_{2,\mu}(\R;H).$
For $\mu\geq\rho$ we compute 
\begin{align*}
\intop_{-\infty}^{a}|\left(T_{\rho}u\right)(t)|\text{ d}t & =\intop_{-\infty}^{a}|\left(T_{\mu}u\right)(t)|\e^{-\mu t}\e^{\mu t}\text{ d}t\\
 & \leq\left|T_{\mu}u\right|_{\mu}\frac{1}{\sqrt{2\mu}}\e^{\mu a}\\
 & \le\frac{1}{\sqrt{2\mu}}|T|_{\mathrm{ev},\rho_{0}}|u|_{\mu}\e^{\mu a}\\
 & =\frac{1}{\sqrt{2\mu}}|T|_{\mathrm{ev},\rho_{0}}\left(\intop_{a}^{\infty}|u(t)|^{2}\e^{-2\mu t}\text{ d}t\right)^{\frac{1}{2}}\e^{\mu a}\\
 & =\frac{1}{\sqrt{2\mu}}|T|_{\mathrm{ev},\rho_{0}}\left(\intop_{a}^{\infty}|u(t)|^{2}\e^{-2\rho t}\e^{2(\rho-\mu)t}\text{ d}t\right)^{\frac{1}{2}}\e^{\mu a}\\
 & \leq\frac{1}{\sqrt{2\mu}}|T|_{\mathrm{ev},\rho_{0}}|u|_{\rho}\e^{\rho a}.
\end{align*}
Letting $\mu\to\infty,$ we infer that $T_{\rho}u=0$ on $]-\infty,a]$
and hence, $\spt T_{\rho}u\subseteq\mathbb{R}_{\geq a}.$ 
\end{proof}
\begin{prop}
\label{prop:causality_aux}Assume \prettyref{hyp:independence} and
let $\delta>\|\mathcal{M}'\|_{\mathrm{ev},\rho_{0}}+\|\mathcal{N}\|_{\mathrm{ev},\rho_{0}.}$
Then, for each $\rho\geq\rho_{0}$ the solution operator $\mathcal{S}_{\rho,\mathrm{aux}}$
is causal. 
\end{prop}

\begin{proof}
Let $\rho\geq\rho_{0}.$ By \prettyref{thm:causaliyt} it suffices
to prove that \prettyref{hyp:causality} hold if we replace $\mathcal{N}_{\rho}$
by $\delta-\mathcal{M}_{\rho}'$. Note that only \prettyref{hyp:causality}
(b) has to be verified for $\mathcal{N}$ replaced by $\delta-\mathcal{M}'.$
So, let $\varphi\in C_{c}^{\infty}(\R;H)$ and $a\in\R.$ Note that
by \prettyref{lem:evo_implies_causal} we have that 
\[
\Re\langle(\partial_{t}\mathcal{M}+\mathcal{N})\varphi,\chi_{\mathbb{R}_{\leq a}}\varphi\rangle_{\rho}\geq c|\chi_{\mathbb{R}_{\leq a}}\varphi|_{\rho}^{2}.
\]
Moreover, we note that by \prettyref{lem:causality_evo} we have 
\[
\chi_{\mathbb{R}_{\leq a}}(\mathcal{M}'+\mathcal{N})\varphi=\chi_{\mathbb{R}_{\leq a}}(\mathcal{M}'+\mathcal{N})\chi_{\mathbb{R}_{\le a}}\varphi.
\]
Hence, we can estimate 
\begin{align*}
\Re\langle(\partial_{t}\mathcal{M}-\mathcal{M}'+\delta)\varphi,\chi_{\mathbb{R}_{\leq a}}\varphi\rangle_{\rho} & =\Re\langle(\partial_{t}\mathcal{M}+\mathcal{N})\varphi,\chi_{\mathbb{R}_{\leq a}}\varphi\rangle_{\rho}+\Re\langle\delta\varphi-(\mathcal{M}'+\mathcal{N})\varphi,\chi_{\R_{\leq a}}\varphi\rangle_{\rho}\\
 & \geq c|\chi_{\mathbb{R}_{\leq a}}\varphi|_{\rho}^{2}+\delta|\chi_{\R_{\leq a}}\varphi|_{\rho}^{2}-\Re\langle\chi_{\R_{\leq a}}(\mathcal{M}'+\mathcal{N})\varphi,\chi_{\R_{\leq a}}\varphi\rangle_{\rho}\\
 & =c|\chi_{\mathbb{R}_{\leq a}}\varphi|_{\rho}^{2}+\delta|\chi_{\R_{\leq a}}\varphi|_{\rho}^{2}-\Re\langle(\mathcal{M}'+\mathcal{N})\chi_{\R_{\leq a}}\varphi,\chi_{\R_{\leq a}}\varphi\rangle_{\rho}\\
 & \geq\left(c+\delta-(\|\mathcal{M}'\|_{\mathrm{ev},\rho_{0}}+\|\mathcal{N}\|_{\mathrm{ev},\rho_{0}})\right)|\chi_{\mathbb{R}_{\leq a}}\varphi|_{\rho}^{2}\\
 & \ge c|\chi_{\mathbb{R}_{\leq a}}\varphi|_{\rho}^{2}.\tag*{\qedhere}
\end{align*}
\end{proof}
With this result at hand, we are able to prove the independence on
the parameter $\rho$ for the solution operator $\mathcal{S}_{\rho,\mathrm{aux}}$.
\begin{prop}
\label{prop:independence_aux}Assume \prettyref{hyp:independence}
and let $\delta>\|\mathcal{M}'\|_{\mathrm{ev},\rho_{0}}+\|\mathcal{N}\|_{\mathrm{ev},\rho_{0}}.$
Then, for each $\mu,\rho\geq\rho_{0}$ and each $f\in L_{2,\rho}(\R;H)\cap L_{2,\mu}(\R;H)$
we have that 
\[
\mathcal{S}_{\rho,\mathrm{aux}}(f)=\mathcal{\mathcal{S}}_{\mu,\mathrm{aux}}(f)
\]
as functions in $L_{1,\mathrm{loc}}(\R;H).$
\begin{proof}
We begin to prove $\mathcal{S}_{\rho,\mathrm{aux}}(f)=\mathcal{S}_{\rho_{0},\mathrm{aux}}(f)$
for all $\rho\geq\rho_{0}$ and $f\in C_{c}^{\infty}(\R;H).$ We set
\begin{align*}
u_{\rho_{0}} & \coloneqq\mathcal{S}_{\rho_{0}}(f)=(\partial_{0,\rho_{0}}\mathcal{M}_{\rho_{0}}-\mathcal{M}_{\rho_{0}}'+\delta+\mathcal{A}_{\rho_{0}})^{-1}(f)
\end{align*}
where we have used \prettyref{cor:maxmonaux}. Let $a\in\R$ such
that $\spt f\subseteq[a,\infty[.$ As $\mathcal{S}_{\rho_{0},\mathrm{aux}}(0)=0$,
since $(0,0)\in\mathcal{A}_{\rho_{0}}$, we derive form the causality
of $\mathcal{S}_{\rho_{0},\mathrm{aux}}$ (see \prettyref{prop:causality_aux})
that $\spt u_{\rho_{0}}\subseteq[a,\infty[$ and hence $u_{\rho_{0}}\in\bigcap_{\mu\geq\rho_{0}}L_{2,\mu}(\R;H).$
Moreover, by \prettyref{lem:causality_evo} and the definition of
$\partial_{t,\rho_{0}}$, $\partial_{t,\rho_{0}}\mathcal{M}_{\rho_{0}}u_{\rho_{0}}\in L_{2,\rho_{0}}(\R;H)$
is supported on $[a,\infty[$ as well and hence, $\partial_{t,\rho_{0}}\mathcal{M}_{\rho_{0}}u_{\rho_{0}}\in L_{2,\rho}(\R;H)$.
The latter gives $\partial_{t,\rho}^{-1}\partial_{t,\rho_{0}}\mathcal{M}_{\rho_{0}}u_{\rho_{0}}=\mathcal{M}_{\rho_{0}}u_{\rho_{0}}=\mathcal{M}_{\rho}u_{\rho_{0}}$
and hence, $\mathcal{M}_{\rho}u_{\rho_{0}}\in\dom(\partial_{t,\rho})$
with $\partial_{t,\rho}\mathcal{M}_{\rho}u_{\rho_{0}}=\partial_{t,\rho_{0}}\mathcal{M}_{\rho_{0}}u_{\rho_{0}}.$
Consequently, 
\[
\left(\partial_{0,\rho_{0}}\mathcal{M}_{\rho_{0}}-\mathcal{M}'_{\rho_{0}}+\delta\right)u_{\rho_{0}}=\left(\partial_{0,\rho}\mathcal{M}_{\rho}-\mathcal{M}'_{\rho}+\delta\right)u_{\rho_{0}}\in L_{2,\rho}(\R;H)
\]
and thus, 
\[
(u_{\rho_{0}},f-\left(\partial_{0,\rho}\mathcal{M}_{\rho}-\mathcal{M}'_{\rho}+\delta\right)u_{\rho_{0}})\in\mathcal{A}_{\rho}.
\]
The latter gives that $u_{\rho_{0}}$ satisfies 
\[
(u_{\rho_{0}},f)\in\partial_{0,\rho}\mathcal{M}_{\rho}-\mathcal{M}_{\rho}'+\delta+\mathcal{A}_{\rho}
\]
and hence $\mathcal{S}_{\rho_{0},\mathrm{aux}}(f)=u_{\rho_{0}}=\mathcal{S}_{\rho}(f).$\\
Let now $f\in L_{2,\rho}(\R;H)\cap L_{2,\mu}(\R;H).$ Then there exists
a sequence $(f_{n})_{n\in\N}$ in $C_{c}^{\infty}(\R;H)$ with $f_{n}\to f$
in $L_{2,\rho}(\R;H)$ and $L_{2,\mu}(\R;H)$ (see e.g. \cite[Lemma 3.5]{Trostorff2013_stability}).
By what we have shown above, we infer 
\[
\mathcal{S}_{\rho,\mathrm{aux}}(f)=\lim_{n\to\infty}\mathcal{S}_{\rho,\mathrm{aux}}(f_{n})=\lim_{n\to\infty}\mathcal{S}_{\mu,\mathrm{aux}}(f_{n})=\mathcal{S}_{\mu,\mathrm{aux}}(f).\tag*{\qedhere}
\]
\end{proof}
\end{prop}

\begin{thm}[Independence of the parameter $\rho$]
\label{thm:independence} Assume \prettyref{hyp:independence}. Then,
for each $\mu,\rho\geq\rho_{0}$ and each $f\in L_{2,\rho}(\R;H)\cap L_{2,\mu}(\R;H)$
we have that 
\[
\mathcal{S}_{\rho}(f)=\mathcal{\mathcal{S}}_{\mu}(f)
\]
as functions in $L_{1,\mathrm{loc}}(\R;H).$
\end{thm}

\begin{proof}
As in the proof of \prettyref{prop:independence_aux} we first show
that for $f\in C_{c}^{\infty}(\R;H)$ we have that 
\[
\mathcal{S}_{\rho}(f)=\mathcal{S}_{\rho_{0}}(f)
\]
for each $\rho\geq\rho_{0}$. For doing so, define $u_{\rho_{0}}\coloneqq\mathcal{S}_{\rho_{0}}(f).$
Note that by \prettyref{thm:causaliyt} and since $\mathcal{S}_{\rho_{0}}(0)=0,$
we derive that $\spt u_{\rho_{0}}$ is bounded from below and hence,
$u_{\rho_{0}}\in L_{2,\rho}(\R;H).$ Thus, by \prettyref{lem:causality_evo}
and \prettyref{prop:independence_aux} we derive that
\begin{align*}
u_{\rho_{0}}=\mathcal{S}_{\rho_{0}}(f) & \Leftrightarrow u_{\rho_{0}}=\mathcal{S}_{\rho_{0},\mathrm{aux}}(f+(\delta-(\mathcal{M}_{\rho_{0}}'+\mathcal{N}_{\rho_{0}})u_{\rho_{0}})\\
 & \Leftrightarrow u_{\rho_{0}}=\mathcal{S}_{\rho,\mathrm{aux}}(f+(\delta-(\mathcal{M}_{\rho}'+\mathcal{N}_{\rho})u_{\rho_{0}})\\
 & \Leftrightarrow u_{\rho_{0}}=\mathcal{S}_{\rho}(f),
\end{align*}
which yields the claim. Since each function $f\in L_{2,\rho}(\R;H)\cap L_{2,\mu}(\R;H)$
can be approximated by one sequence of functions in $C_{c}^{\infty}(\R;H)$
in both spaces $L_{2,\rho}(\R;H)$ and $L_{2,\mu}(\R;H),$ the assertion
of the theorem follows.
\end{proof}

\section{An Application to semistatic quasilinear Maxwell's equations }

We consider the following variant of Maxwell's equations on a domain
$\Omega\subseteq\mathbb{R}^{3}$ in the semistatic case 
\begin{align}
\sigma E-\curl H & =-J,\nonumber \\
\partial_{t,\rho}B+\curl E & =0,\label{eq:Maxwell}
\end{align}
where $(E,H)\in L_{2,\rho}(\R;L_{2}(\Omega)^{3})\times L_{2,\rho}(\R;L_{2}(\Omega)^{3})$
denotes the electro-magnetic field and $B\in L_{2,\rho}(\R;L_{2}(\Omega)^{3})$
is the magnetic induction and $J\in L_{2,\rho}(\R;L_{2}(\Omega)^{3})$
is a given current density. The conductivity of the underlying medium
is denoted by $\sigma$ and will be specified later. The equations
are completed by the following constitutive relation linking $H$
and $B$
\begin{equation}
(H(t),B(t))\in\mathcal{Z}\quad(t\in\R\text{ a.e.})\label{eq:const}
\end{equation}
for some $c$-maximal monotone bounded relation $\mathcal{Z}\subseteq L_{2}(\Omega)^{3}\times L_{2}(\Omega)^{3}$
with $c>0$ and the boundary conditions 
\begin{align}
n\cdot\left(J+\sigma E\right) & =0,\text{ on }\partial\Omega\nonumber \\
n\cdot B & =0,\text{ on \ensuremath{\partial\Omega}},\label{eq:bd}
\end{align}
where $n$ denotes the outward unit normal vector on $\partial\Omega.$
It will turn out that in general the conditions \prettyref{eq:bd}
are not sufficient to ensure the uniqueness of a solution. We will
come to this point later.

We start by introducing some variants of the operators $\curl$ and
$\dive$.
\begin{defn*}
We define the operator $\dive_{0}$ as the closure of 
\[
C_{c}^{\infty}(\Omega)^{3}\subseteq L_{2}(\Omega)^{3}\to L_{2}(\Omega):\:(\psi_{i})_{i\in\{1,2,3\}}\mapsto\sum_{i=1}^{3}\partial_{i}\psi_{i}.
\]
In the same way, we define the operator $\curl_{0}$ as the closure
of 
\[
C_{c}^{\infty}(\Omega)^{3}\subseteq L_{2}(\Omega)^{3}\to L_{2}(\Omega)^{3}:\:(\psi_{i})_{i\in\{1,2,3\}}\mapsto(\partial_{2}\psi_{3}-\partial_{3}\psi_{2},\partial_{3}\psi_{1}-\partial_{1}\psi_{3},\partial_{1}\psi_{2}-\partial_{2}\psi_{1})^{\top}.
\]
Moreover, we define $\curl\coloneqq\left(\curl_{0}\right)^{\ast}.$ 
\end{defn*}
\begin{rem}
In case of a smooth boundary $\partial\Omega$ the elements in $\Psi\in\dom(\dive_{0})$
can be characterised as those $L_{2}$-vector fields, whose divergence
is an $L_{2}$-function and which satisfy the homogeneous Neumann
boundary condition 
\[
\Psi\cdot n=0.
\]

Note that $\Psi\in\dom(\dive_{0})$ also makes sense even if $\Omega$
has a non-smooth boundary, where the classical homogeneous Neumann
boundary condition could not be formulated via trace theorems. That
is why we do not need to require any regularity on $\Omega$ and may
use $\Psi\in\dom(\dive_{0})$ as a suitable generalised Neumann boundary
condition.

In the same way, the elements in $E\in\dom(\curl_{0})$ are those
$L_{2}$-vector fields with a curl representable as an $L_{2}$-vector
field and who satisfy the (generalised) electrical boundary condition
\[
E\times n=0.
\]
Moreover, by definition 
\[
\dom(\curl)=\{E\in L_{2}(\Omega)^{3}\,;\,\curl E\in L_{2}(\Omega)^{3}\},
\]
i.e., $\curl$ is the maximal realisation of the rotation as an operator
on $L_{2}(\Omega)^{3}.$ 
\end{rem}

Let $(E,B,H)$ be a solution of \prettyref{eq:Maxwell},\prettyref{eq:const}
and \prettyref{eq:bd} . From the first line of \prettyref{eq:Maxwell}
we read of that $\sigma E+J$ has a vanishing $L_{2}$-divergence
and since $\sigma E+J$ should satisfy \prettyref{eq:bd}, we infer
that indeed $E\in\ker(\dive_{0}).$ In the same way, $B\in\ker(\dive_{0}).$
We note that due to 
\[
\dive_{0}\curl_{0}=0
\]
it follows that $\overline{\ran(\curl_{0})}$ is a closed subspace
of $\ker(\dive_{0}).$ Hence, according to the projection theorem,
we can decompose $\ker(\dive_{0})$ by (see also \cite{MilaniPicard1988,Picard1998})
\begin{align}
\ker(\dive_{0}) & =\overline{\ran(\curl_{0})}\oplus\left(\ran(\curl_{0})\right)^{\bot}\cap\ker(\dive_{0})\nonumber \\
 & =\overline{\ran(\curl_{0})}\oplus\ker(\curl)\cap\ker(\dive_{0}).\label{eq:Helmholtz}
\end{align}
The set 
\[
H_{N}\coloneqq\ker(\curl)\cap\ker(\dive_{0})
\]
is known as the set of harmonic Neumann fields. Note that $H_{N}$
is a closed subspace of $L_{2}(\Omega)^{3}$ and we denote the orthogonal
projection onto $H_{N}$ by $P_{N}$. Additionally to the boundary
conditions \prettyref{eq:bd} we need to impose conditions on the
values $P_{N}(\sigma E+J)$ and $P_{N}B.$ For simplicity we set
\begin{equation}
P_{N}(\sigma E+J)=P_{N}B=0.\label{eq:constraints}
\end{equation}

For later reference, we summarise our so far found constraints: 
\begin{align}
\sigma E-\curl H & =-J,\nonumber \\
\partial_{t,\rho}B+\curl E & =0,\nonumber \\
(H(t),B(t)) & \in\mathcal{Z}\quad(t\in\R\text{ a.e.}),\nonumber \\
\sigma E+J,B & \in\ker(\dive_{0})\cap H_{N}^{\bot}.\label{eq:MAXWELL}
\end{align}

In order to incorporate the conditions \prettyref{eq:bd} and \prettyref{eq:constraints},
we need the following variant of the $\curl$ operator.
\begin{prop}[{\cite[Theorem 2.1]{Picard1998}}]
\label{prop:selfadjointness_curl} Assume that $\ran(\curl_{0})$
is closed. Then the operator 
\[
\tilde{\curl}:\dom(\tilde{\curl})\subseteq L_{2}(\Omega)^{3}\to L_{2}(\Omega)^{3}:\;W\mapsto\curl W
\]
with $\dom(\tilde{\curl})\coloneqq\left\{ W\in\dom(\curl)\,;\,\curl W\in\ran(\curl_{0})\right\} $
is selfadjoint.
\end{prop}

The assumption that $\ran(\curl_{0})$ is closed can for instance
be ensured by assuming the compactness of the embedding 
\[
\dom(\curl_{0})\cap\dom(\dive)\hookrightarrow L_{2}(\Omega)^{3}.
\]
We refer to \cite{Weck1974,Picard1984} for a proof of this result
even on a class of Riemannian manifolds with non-smooth boundaries
and to \cite{Witsch1993} for the case of a bounded domain whose boundary
is allowed to have certain cusps. 

We now come to the well-posedness result.
\begin{thm}
Let $\sigma,\sigma',\kappa:C_{c}^{\infty}(\R;L_{2}(\Omega)^{3})\to\bigcap_{\mu\geq\rho_{0}}L_{2,\mu}(\R;L_{2}(\Omega)^{3})$
be evolutionary at $\rho_{0}>0$ such that 
\[
(\sigma\varphi)'=\sigma\varphi'+\sigma'\varphi\quad(\varphi\in C_{c}^{\infty}(\R;H))
\]
and there exists $c_{1}>0$ such that
\[
\Re\langle\left(\partial_{t,\rho}\sigma+\kappa\right)\varphi,\chi_{\R_{\leq a}}\varphi\rangle_{\rho}\geq c_{1}|\chi_{\R_{\leq a}}\varphi|_{\rho}^{2}\quad(\varphi\in C_{c}^{\infty}(\R;H),a\in\R).
\]
Moreover, assume that $\ran(\curl_{0})\subseteq L_{2}(\Omega)^{3}$
is closed. Finally, let $\mathcal{Z}\subseteq L_{2}(\Omega)^{3}\times L_{2}(\Omega)^{3}$
be $c$-maximal monotone and bounded with $(0,0)\in\mathcal{Z}$.
Then for each $\rho\geq\rho_{0}$ 
\[
\mathcal{S}_{\rho,\mathrm{Max}}\coloneqq\left(\overline{\partial_{t,\rho}\left(\begin{array}{cc}
\sigma_{\rho} & 0\\
0 & 0
\end{array}\right)+\left(\begin{array}{cc}
\kappa_{p} & 0\\
0 & c
\end{array}\right)+\left(\begin{array}{cc}
0 & -\tilde{\curl}\\
\tilde{\curl} & \mathcal{Z}-c
\end{array}\right)_{\rho}}\right)^{-1}
\]
is a Lipschitz-continuous, causal mapping on $L_{2,\rho}(\R;L_{2}(\Omega)^{3}\times L_{2}(\Omega)^{3})$
and independent of $\rho$ in the sense of \prettyref{thm:independence}.
\end{thm}

\begin{proof}
As $\tilde{\curl}$ is selfadjoint, we infer that 
\[
\left(\begin{array}{cc}
0 & -\tilde{\curl}\\
\tilde{\curl} & 0
\end{array}\right)
\]
is skew-selfadjoint and hence, maximal monotone. Moreover $\mathcal{Z}-c$
is maximal monotone by \prettyref{lem:lifting} and bounded and thus,
\[
\mathcal{A}\coloneqq\left(\begin{array}{cc}
0 & -\tilde{\curl}\\
\tilde{\curl} & \mathcal{Z}-c
\end{array}\right)
\]
is maximal monotone by \prettyref{cor:bounded_pert}. Moreover, $(0,0)\in\mathcal{A}$
and hence, the operators $\mathcal{M}\coloneqq\left(\begin{array}{cc}
\sigma & 0\\
0 & 0
\end{array}\right),\mathcal{N}\coloneqq\left(\begin{array}{cc}
\kappa & 0\\
0 & c
\end{array}\right)$ and the relation $\mathcal{A}$ satisfy \prettyref{hyp:independence}.
Thus, the assertion follows by \prettyref{thm:main}, \prettyref{thm:causaliyt}
and \prettyref{thm:independence}.
\end{proof}
\begin{rem}
\begin{enumerate}[(a)]

\item A particular instance of operators $\sigma,\kappa$ satisfying
the assumptions of the previous theorem are multiplication operator
associated with a Lipschitz-continuous, bounded mapping $\sigma,\kappa:\R\to L(L_{2}(\Omega)^{3})$
such that $\sigma(t)$ is injective, selfadjoint and uniformly strictly
positive definite for each $t\in\R$. Those operators were considered
in \cite{Picard2013_nonauto,Trostorff2013_nonautoincl}.

\item In case of $\mathcal{Z}$ being the inverse of a Lipschitz-continuous,
$c$-monotone mapping, $\sigma$ a positive constant number and $\kappa=0$,
the well-posedness result was already obtained in \cite{Milani_Picard1995}
employing a Galerkin approximation method.

\end{enumerate}
\end{rem}

To conclude this section, we show how the mapping $\mathcal{S}_{\rho,\mathrm{Max}}$
indeed (at least formally\footnote{By assuming enough regularity of $J$, the computations could be made
rigorous.}) provides a solution of our problem \prettyref{eq:MAXWELL}.

Assume that $(E,B,H)$ satisfy \prettyref{eq:MAXWELL}. As $B$ and
$\sigma E+J$ lie in $\ker(\dive_{0})\cap H_{N}^{\bot}$ by \prettyref{eq:bd}
and \prettyref{eq:constraints}, we derive from \prettyref{eq:Helmholtz}
that 
\[
B,\sigma E+J\in\ran(\curl_{0}).
\]
However, 
\[
\partial_{t,\rho}B=-\curl E,\:\sigma E+J=\curl H
\]
and thus, $H,E\in\dom(\tilde{\curl}).$ Thus, we have the equations
\begin{align*}
\sigma E-\tilde{\curl}H & =-J,\\
\partial_{t,\rho}B+\tilde{\curl}E & =0.
\end{align*}
Defining $\tilde{E}\coloneqq\partial_{t,\rho}^{-1}E$ we get 
\begin{align*}
\partial_{t,\rho}\sigma\tilde{E}-\sigma'\tilde{E}-\tilde{\curl}H & =-J,\\
B+\tilde{\curl}\tilde{E} & =0,
\end{align*}
where we have used $\sigma\partial_{t,\rho}\tilde{E}=\partial_{t,\rho}\sigma\tilde{E}-\sigma'\tilde{E}.$
Invoking \prettyref{eq:const}, we derive that
\[
\left((\tilde{E},H),(-J,0)\right)\in\partial_{t,\rho}\left(\begin{array}{cc}
\sigma_{\rho} & 0\\
0 & 0
\end{array}\right)+\left(\begin{array}{cc}
-\sigma'_{\rho} & 0\\
0 & c
\end{array}\right)+\left(\begin{array}{cc}
0 & -\tilde{\curl}\\
\tilde{\curl} & \mathcal{Z}-c
\end{array}\right)_{\rho}.
\]
On the other, following this argumentation backwards, we infer that
(for $\kappa=-\sigma'$) 
\[
(\tilde{E},H)\coloneqq\mathcal{S}_{\rho,\mathrm{Max}}(-J,0)
\]
yields a solution $(E,B,H)$ of \prettyref{eq:MAXWELL} with $E\coloneqq\partial_{t,\rho}\tilde{E}$
and $B=-\tilde{\curl}\tilde{E}.$

\end{document}